\documentclass{article}
\usepackage[%
journal=EMSS,    %% Replace XXX by one of the above journal codes.
lang=american,   %% Change to `american' if you use American English.
]{ems-journal}

\theoremstyle{plain}
\newtheorem{thm}{Theorem}[section]
\newtheorem{theorem}[thm]{Theorem}

\newtheorem{lemma}[thm]{Lemma}
\newtheorem{proposition}[thm]{Proposition}
\newtheorem{corollary}[thm]{Corollary}
\newtheorem{conjecture}[thm]{Conjecture}

\theoremstyle{definition}
\newtheorem{definition}[thm]{Definition}
\newtheorem{remark}[thm]{Remark}
\newtheorem{example}[thm]{Example}

\newtheorem{problem}[thm]{Problem}

\usepackage{pgfplots}
\pgfplotsset{compat=1.15}
\usepackage{mathrsfs}
\usetikzlibrary{arrows}

\newtheorem{thevarthm}[thm]{\varthmname}

\newenvironment{varthm*}[1]{\trivlist\item[]{\bf #1.}\it}{\endtrivlist}

\DeclareMathOperator{\Spec}{Spec}

\usepackage{xurl}
\usepackage{rotating}
\numberwithin{equation}{section}

\begin{document}

%------
\title{On the geography of log-surfaces}
\titlemark{On the geography of log-surfaces}

%------

%%%% Pls fill in all fields for each author
%%%% Label the authors by their position in the authors' list using {}
%%%% If you published any math paper ever, you have an MR Author ID.
%  Please look it up in three easy (and free) steps:
% 1. copy the bibliographic data of any published paper (co-)authored by you in the search field at https://mathscinet.ams.org/mathscinet/freetools/mref
% 2. Hit your name in the search result
% 3. Find your MR Author ID in the first row, copy it in the \mrid{} field
%%%% If you have not created your ORCID yet, you may like to do it now, pls copy it in the field \orcid{}
%%%% Abbreviate first names for the running head

\emsauthor{1}{
	\givenname{Bartosz}
	\surname{Naskr\k{e}cki}
	\mrid{917534}
	\orcid{0000-0003-2484-143X}}{B.~Naskr\k{e}cki}
%%%% Repeat the same fields for each numbered author
\emsauthor{2}{
	\givenname{Piotr}
	\surname{Pokora}
	\mrid{999250}
	\orcid{0000-0001-8526-9831}}{P.~Pokora}

%%%% Please provide detailed address info for each author
%%%% Use the same numbering as for \emsauthor above
%%%% Please look up the ROR ID of your institute here: https://ror.org
\Emsaffil{1}{
	\department{ Faculty of Mathematics and Computer Science}
	\organisation{Adam Mickiewicz University Pozna\'n}
	\rorid{}
	\address{Uniwersytetu Pozna\'nskiego 4}
	\zip{PL-61-614}
	\city{Pozna\'n}
	\country{Poland}
	\affemail{bartosz.naskrecki@amu.edu.pl}}
%%%% Repeat the same fields for each numbered author
%%%% If some author has multiple affiliations, repeat the fields for each affiliation
%%%% Number the affiliations using {}
\Emsaffil{2}{
	\department{1}{Department of Mathematics}
	\organisation{1}{University of the National Education Commission Krakow}
	\rorid{1}{}
	\address{1}{Podchor\c a\.zych 2}
	\zip{1}{PL-30-084}
	\city{1}{Krak\'ow}
	\country{1}{Poland}
	\affemail{piotr.pokora@uken.krakow.pl}}

%------
% Add MSC 2020 codes according to https://zbmath.org/classification/.
% A unique primary MSC code (in curly brackets) is mandatory,
% while secondary MSC codes (in square brackets) are optional.
%------
\classification[14C20]{14J29, 14J28, 14N25}

%------
% Add a list of keywords.
%------

%------
% Insert your abstract.
%------
\begin{abstract}
This survey focuses on the geometric problem of log-surfaces, which are pairs consisting of a smooth projective surface and a reduced non-empty boundary divisor. In the first part, we focus on the geography problem for complex log-surfaces associated with pairs of the form $(\mathbb{P}^{2}, C)$, where $C$ is an arrangement of smooth plane curves admitting ordinary singularities. Specifically, we focus on the case in which $C$ is an arrangement consisting of smooth rational curves as its irreducible components. In the second part, containing original new results, we study log-surfaces constructed as pairs consisting of a complex projective $K3$ surface and a rational curve arrangement. In particular, we provide some combinatorial conditions for such pairs to have the log-Chern slope equal to $3$. Our survey is illustrated with many explicit examples of log-surfaces.
\end{abstract}

\maketitle
\section{Introduction}

The geography problem is one of the oldest problems in the theory of complex algebraic surfaces of a very fundamental nature. Let us recall the most basic formulation that fits into our presentation.
\begin{problem}
Let $\mathcal{M}_{c_{1}^{2}, c_{2}}$ be the space of parameters for minimal complex surfaces of general type $X$ with $c_{1}^{2}:=c_{1}^{2}(X)$ and $c_{2}:=c_{2}(X)$ fixed. For which values of $c_{1}^{2}$ and $c_{2}$ are the parameter spaces non-empty?
\end{problem}
Let us recall that Gieseker in \cite{G77} proved that the parameter space of minimal surfaces of general type with fixed Chern numbers (of the canonical model) exists as a quasi-projective variety.
It is well-known that there are some bounds restricting possible values for the pairs $(c_{1}^{2},c_{2})$, and these results can be attributed to several mathematicians, including Castelnuovo \cite{Cast}, Noether \cite{MN}, Bogomolov \cite{Bog}, Miyaoka \cite{Miyaoka77} and Yau \cite{Yau}. More precisely, we have the following.
\begin{thm}
Let $X$ be a minimal complex projective surface of general type. Then one has the following constraints:

\begin{description}
    \item[\textbf{Castelnuovo}:] $c_{1}^{2}, c_{2} \geq 1$;
    \item[\textbf{Noether}:] $\frac{1}{5}(c_{2}-36)\leq c_{1}^{2}$ and $c_{1}^{2} + c_{2} \equiv 0 \, (\textrm{mod} \, 12)$;
    \item[\textbf{Bogomolov--Miyaoka--Yau}:] $c_{1}^{2}\leq 3c_{2}$.
\end{description}
\end{thm}
The last constraint, called the BMY--inequality for short, provides a numerical characterization of the so-called ball-quotient surfaces, namely $c_{1}^{2}(X) = 3c_{2}(X)$ if and only if $X$ is a quotient of the unit complex ball. 

The first natural question that arises is how to find a surface of general type with fixed Chern numbers. As one might predict, this is generally a difficult problem. The difficulty of the problem increases as the Chern slope (or Chern ratio), defined as $c_{1}^{2}(X)/c_{2}(X)$, increases. It is not difficult to find smooth projective surfaces such that $c_{1}^{2}(X) \leq 2c_{2}(X)$. In that context, we have the following interesting result that comes from \cite{Persson}.
\begin{theorem}[Persson]
For any pair of positive integers $(c_{1}^{2},c_{2})$ satisfying   $\frac{1}{5}(c_{2}-36)\leq c_{1}^{2} \leq 2c_{2}$ and $c_{1}^{2} + c_{2} \equiv 0\, (\textrm{mod} \, 12)$, there exists a minimal smooth projective surface $X$ of general type with $c_{1}^{2} = c_{1}^{2}(X)$ and $c_{2}=c_{2}(X)$.  
\end{theorem}
The case where $2c_{2} < c_{1}^{2} \leq 3c_{2}$ is much harder. It is worth noting that Hirzebruch in \cite{Hirzebruch} explicitly constructed examples of ball-quotient surfaces using the full abelian cover of the complex projective plane branched at arrangements of lines. Based on the ideas delivered by Hirzebruch, Sommese in \cite{Sommese} proved the following density result.
\begin{theorem}\label{thm:dense_limits}
Any rational point in $[1/5, 3]$ occurs as a limit of Chern slopes of smooth projective surfaces of general type.
\end{theorem}

\begin{remark}
Since $c_1^2(X)$ and $c_2(X)$ are integers, Chern slopes $c_1^2(X)/c_2(X)$ of smooth surfaces of general type are rational numbers. Theorem \ref{thm:dense_limits} asserts that \textit{every} rational number in $[1/5,3]$ appears as a limit of such slopes. In particular, the set of realized slopes is dense in $[1/5,3]$. See also Section \ref{sec:line_arr} for explicit families obtained from abelian covers. In Example \ref{ex:limit_eight_three} we present how $8/3$ can be obtained as a proper limit of a sequence of slopes.
\end{remark}

Now we would like to explain the importance of Hirzebruch's paper \cite{Hirzebruch} for the development of log-surfaces, namely his idea to study the so-called characteristic numbers of line arrangements. These numbers are nothing else than log-Chern slopes of the associated log-surfaces $(X, \widetilde{\mathcal{L}})$, where $X$ denotes the blowing up of the complex projective plane along the intersection points $p \in \textrm{Sing}(\mathcal{L})$ of a given line arrangement $\mathcal{L} \subset \mathbb{P}^{2}$ with multiplicity $m_{p}\geq 3$, and $\widetilde{\mathcal{L}}$ is the reduced total transform of $\mathcal{L}$. In particular, Hirzebruch showed that for a real line arrangement $\mathcal{L} \subset \mathbb{P}^{2}$, i.e., a line arrangement defined over the real numbers, its characteristic numbers is $\leq 5/2$, and equality holds if and only if $\mathcal{L}$ is simplicial. This result builds a foundation of our survey and our research. Using Hirzebruch's idea, we focus on a special variant of the geography problem, namely \textbf{we work with log-surfaces for which their log Chern numbers are directly determined by the combinatorics of the curve arrangements}. This setting, if it looks at the first glance quite restrictive, is not well explored, and we still do not know much about the geography of such surfaces. 

Here is the outline of our survey which is written in the spirit of \cite{ML}. The first part is a solid preparation where we explain all necessary details regarding the geometry of log-surfaces. Then we pass to geography problem where we start with the case when our building blocks are line arrangements in the complex projective plane as this is the most classical case. We provide an outline of the state of the art and we formulate some very challenging problems. Then we focus on the case when our building blocks are arrangements of smooth rational curves that admit only ordinary singularities. In particular, we outline results devoted to conic-line arrangements in the complex projective plane. In the last part of our survey, we focus on log-surfaces constructed using arrangements of rational curves admitting only ordinary singularities in $K3$ surfaces. We hope that results presented in our survey will deliver a platform for further research devoted to the geography problem of log-surfaces.

We work exclusively over the complex numbers.
\section{Basic notions from log geometry}

We start with the following fundamental definition.
\begin{definition}
Let $X$ be a smooth complete variety and let $D \subset X$ be a simple normal crossing divisor. A log variety is a smooth variety $U$ which is of the form $U = X \setminus D$. 
\end{definition}
Usually one refers to $U$ as the pair $(X,D)$. We call it log-surface if its dimension is equal to two. 

Let $(X,D)$ be a log variety. In order to study the geography problem for such objects, we need to define the Chern numbers of the log variety $(X,D)$. First of all, the analogue of a canonical divisor will be $K_{X}+D$, and this idea comes from the following modification of the sheaf of differentials that was proposed by Deligne \cite{Deligne}. 
\begin{definition}
Let $X$ be a smooth complete variety defined over an algebraically closed field $k$ and let $D = \sum_{j=1}^{n}D_{j}$ be a reduced normal crossing divisor, i.e., a divisor with non-singular components $D_{j}$ intersecting each other transversally. Define $\tau : U = X \setminus D \rightarrow X$ and
\[\Omega^{k}_{X}(\star D) = \mathop{\lim_{\longrightarrow}}_{v}\Omega_{X}^{k}(v\cdot D) =  \tau_{*}\Omega^{k}_{U}.\] 
We recall that $\Omega^{\bullet}_{X}(\star D)$ with the differential $\partial$ creates a complex.

Then $\Omega^{k}_{X}(\textrm{log}\,D)$ is the subsheaf of $\Omega_{X}^{k}(\star D)$ of differential forms with logarithmic poles along $D$, i.e., if $V \subseteq X$ is open, then
\[\Gamma(V, \Omega^{k}_{X}(\textrm{log}\,D)) = \{\alpha \in \Gamma(V, \Omega^{k}_{X}(\star D): \alpha \text{ and } \partial\alpha \text{ have simple poles along } D\}.\]
\end{definition}
\begin{lemma}[{\cite[2.2. Properties]{EV}}]
Let $X$ be a smooth complete variety and let $D \subset X$ be a simple normal crossing divisor. Then the sheaf of logarithmic differentials along $D$, denoted by $\Omega_{X}^{1}(\textrm{log} \, D)$, is an $\mathcal{O}_{X}$-submodule of $\Omega_{X}^{1} \otimes \mathcal{O}_{X}(D)$ that satisfies the following conditions:
\begin{itemize}
\item $\Omega_{X}^{1}(\textrm{log} \, D)_{|_{X\setminus D}} = \Omega^{1}_{X\setminus D}$,
\item for every closed point $P \in D$ one has
\[\omega_{P} \in \Omega_{X}^{1}(\textrm{log} \, D)_{p} \text{ if and only if } \omega_{P} = \sum_{i=1}^{s} a_{i}\frac{dz_{i}}{z_{i}} + \sum_{j=s+1}^{\textrm{dim} \,X} b_{j}dz_{j},\]
where $(z_{1}, \ldots, z_{\textrm{dim}X})$ is a local coordinate system around $P$ for $X$, and $\{z_{1} \cdots z_{s} = 0\}$ defines $D$ around point $P$.
\end{itemize}
\end{lemma}
Notice that $\Omega_{X}^{1}(\textrm{log}\,D)$ is a locally free sheaf of rank $\textrm{dim} \, X$ and $\Omega_{X}^{\textrm{dim}\, X}(\textrm{log}\,D) \cong \mathcal{O}_{X}(K_{X}+D)$.

\begin{definition}
The logarithmic Kodaira dimension of $(X,D)$ is defined as the maximum dimension of the image of $|m(K_{X}+D)|$ for all $m > 0$, or is $-\infty$ if $|m(K_{X}+S)|$ is empty for all $m$. We will denote the log Kodaira dimension of $(X,D)$ as $\overline{\kappa}(X,D)$.
\end{definition}
\begin{definition}
The $m$-th log plurigenus of a log variety $(X,D)$ is defined as $\overline{P}_{m}(X,D) = h^{0}(X, m(K_{X}+D))$.
\end{definition}
The following result is a fundamental one for log-varieties.
\begin{theorem}[{\cite[p. 60]{MM02}}]
Let $X_{0}$ be a smooth non-complete variety. Let $(X,D)$ be a log variety such that $X_{0} \cong X \setminus D$. Then the $m$-th logarithmic plurigenus of $(X,D)$ is a well-defined invariant of $X_{0}$, i.e., it means that for another such pair $(X',D')$ we have $\overline{P}_{m}(X,D) = \overline{P}_{m}(X',D')$. This implies that the logarithmic Kodaira dimension of $X_{0}$ is well-defined.
\end{theorem}
Now we want to define the most important, from our perspective, invariants of log varieties.
\begin{definition}
\label{chgen}
The logarithmic Chern classes of the log variety $(X,D)$ are defined as
\[\overline{c}_{i}(X,D) = c_{i}(\Omega_{X}^{1}(\textrm{log}\,D)^{\vee}) \text{ for } 1 \leq i \leq \textrm{dim}\, X.\]
\end{definition}
After such a general preparation, we now focus on the case of surfaces.

Let $(X,D)$ be a log-surface, i.e., $X$ is a smooth projective surface and $D = \sum_{i=1}^{n} D_{i}$, where $D_{i}$ are smooth projective curves and $D$ has at most ordinary double points as singularities. The Chern numbers of such a log-surface $(X,D)$ are the following:
\begin{equation}
\label{chnumb}
\overline{c}_{1}^{2}(X,D) = (K_{X}+D)^2 \text{ and } \overline{c}_{2}(X,D) = e(X) - e(D),
\end{equation}
where $e(\cdot)$ denotes the holomorphic Euler characteristic.
\begin{remark}
In the case of surfaces, it is worth noting that Definition \ref{chgen} simplifies to \eqref{chnumb}. Indeed, we write $E:=\Omega_X^1(\log D)$. First, $\det E\simeq \mathcal{O}_X(K_X+D)$, hence
$$c_1(E^\vee)=-c_1(E)=-(K_X+D),$$ and therefore $c_1^2(E^\vee)=(K_X+D)^2$. For $c_2$, we can use the exact sequence 
\[0\to \Omega_X^1 \to \Omega_X^1(\log D)\to \oplus_i \mathcal{O}_{D_i}\to 0\]
and multiplicativity of total Chern classes to get
\begin{align*}
c_1(E)&=K_X+D,\\
c_2(E)&=c_2(X)+K_X\cdot D+\sum_i D_i^2+\sum_{i<j}D_i\cdot D_j.
\end{align*}
Since $e(X)=c_2(X)$, and for smooth curves $D_i$ we have $e(D_i)=-(K_X+D_i)\cdot D_i$ by adjunction, and inclusion--exclusion for an SNC divisor with only double points gives \[e(D)=\sum_i e(D_i)-\sum_{i<j}D_i\cdot D_j= -K_X\cdot D-\sum_i D_i^2-\sum_{i<j}D_i\cdot D_j.\]
Hence $c_2(E)=e(X)-e(D)$. Finally, recall that $c_2(E^\vee)=c_2(E)$ for rank‑two bundles.
\end{remark}
Form the perspective of our purposes in this survey, the following definition is fundamental.
\begin{definition}
The log Chern slope of a given log-surface $(X,D)$ is defined as
\[E(X,D) = \overline{c}_{1}^{2}(X,D) / \overline{c}_{2}(X,D).\]
\end{definition}
It is natural to ask whether there are some constraints on the log Chern slopes that are analogous to those proved for smooth complex projective surfaces. We have the following result due to Sakai \cite{Sakai}. Let us recall that a divisor $D \subset X$ is \textit{semi-stable} if $D$ is a simple normal crossing divisor and any $\mathbb{P}^{1}$ contained in $D$ intersects other components in more than one point. 
\begin{theorem}[Sakai]
Let $(X,D)$ be a log-surface with $D$ being semi-stable. Suppose that the logarithmic Kodaira dimension of $(X,D)$ is equal to $2$, then
\[\overline{c}_{1}^{2}(X,D) \leq 3 \overline{c}_{2}(X,D).\]
\end{theorem}
Since the above result relies heavily on Miyaoka's proof of his inequality \cite{M84}, we refer to this result as the Miyaoka--Sakai inequality. 

Up to now, our log-surfaces are constructed by taking suitable divisors which are arrangements of curves with only nodes as singularities. Now we would like to introduce a natural construction that allows us to use \textbf{any} arrangement of smooth curves admitting ordinary singularities.
\begin{definition}\label{def:trans}
Let $X$ be a smooth complex projective surface. We say that an arrangement of curves $\mathcal{D} = \{C_{1}, \ldots, C_{n}\}$ in $X$ is a transversal arrangement of curves if
\begin{itemize}
\item all components $C_{i}$ are smooth and irreducible curves;
\item the curves $C_{i}, C_{j}$ with $i\neq j$ are either disjoint or intersect transversally;
\item the arrangement $\mathcal{D}$ is connected, i.e., the union $C_{1} \cup \ldots \cup C_{n}$ is a connected curve in~$X$.
\end{itemize}
\end{definition}
Note that our definition \textbf{does not exclude} the possibility that more than two curves meet in a point. In other words, in such an arrangement of curves all the intersection points are \textbf{ordinary singularities}.

\begin{figure}[h]
\centering
\begin{tikzpicture}[line cap=round,line join=round,>=triangle 45,x=1.0cm,y=1.0cm,scale=0.15]
\clip(-21,-12) rectangle (26,13);
\draw [line width=1.pt] (0.0,-12.00) -- (0.,12.68);
\draw [line width=1.pt,domain=-21.0:26] plot(\x,{(-0.-0.*\x)/1.});
\draw [line width=1.pt] (3.0,-12.00) -- (3.,12.68);
\draw [line width=1.pt,domain=-21.0:26] plot(\x,{(--3.-1.*\x)/-1.});
\end{tikzpicture}
\caption{A transversal arrangement $\mathcal{D}=\{C_1,\dots,C_n\}$ of smooth curves on a smooth surface $X$: any two distinct components are either disjoint or meet transversally, several curves may meet at one point (ordinary $k$-fold points with $k\ge2$), and the union $\bigcup_i C_i$ is connected. Blowing up the points with $k\ge3$ yields the associated log surface $(Y,\widetilde D)$ with $\widetilde D$ a simple normal crossings divisor (cf. Definition \ref{def:trans} and Proposition \ref{chernn}).}
\end{figure}

Denote by $D$ the associated divisor, i.e., $D=\sum_{i=1}^{n} C_{i}$, and we denote by $t_{k}(\mathcal{D})=t_{k}$ the number of $k$-fold points of $\mathcal{D}$, i.e., points in $X$ where exactly $k$ curves from $\mathcal{D}$ meet. The last condition about the connectedness of our arrangement is a technical assumption formulated to avoid possible pathological situations.

Let $(X,D)$ be a pair with $X$ being a smooth complex projective surface and $D$ be a divisor associated to a transversal arrangement of curves $\mathcal{D}$\footnote{For simplicity, we will identify transverse arrangements of curves with their corresponding divisors}. Now we associate with the pair $(X,D)$ a log-surface. To this end, let $\sigma : Y \longrightarrow X$ be the blowing-up of $X$ at all $k$-fold points of $\mathcal{D}$ with $k\geq 3$ and let $\widetilde{D}$ be the reduced total transform of $D$ under $\sigma$. Hence, $\widetilde{D}$ includes the exceptional divisors over the $k\geq 3$-fold points. Consider now $\widetilde{D}$ as an arrangement contained in $Y$ and observe that $\widetilde{D}$ is a simple normal crossing divisor. Then $(Y,\widetilde{D})$ is a log-surface associated with the pair $(X, D)$. 
\begin{proposition}[{\cite[p. 27]{Urzua}}]
\label{chernn}
In the setting as above, let $(Y, \widetilde{D})$ be a log-surface associated with the pair $(X,D)$. The log Chern numbers of the pair are the following:
\begin{equation}
\overline{c}_{1}^{2}(Y,\widetilde{D}) = c_{1}^{2}(X) - \sum_{i=1}^{n}C_{i}^{2} + \sum_{k \geq 2} (3k-4)t_{k} + 4\sum_{i=1}^{n}(g(C_{i})-1),
\end{equation}
\begin{equation}
\overline{c}_{2}(Y,\widetilde{D}) = c_{2}(X) + \sum_{k\geq 2}(k-1)t_{k} + 2\sum_{i=1}^{n}(g(C_{i})-1),
\end{equation}
where $g(C)$ denotes the geometric genus of a curve $C$.
\end{proposition}

We will be interested in extremal arrangements, both in the sense of combinatorics and in our desire to have the log Chern slopes of $(Y, \widetilde{D})$ as close to $3$ as possible, according to the Miyaoka--Sakai inequality. The most important problem in the geography of log-surfaces is to decide whether the Miyaoka--Sakai inequality is really optimal, i.e. whether we can find a log-surface with a log Chern slope as close to $3$ as possible. Based on our practical experience, we know that this is an extremely difficult problem. In almost all of the cases we studied, it was impossible. It is a good moment to recall the following (folkloric) problem, which serves as our leitmotif.
\begin{problem}
\label{main}
Let $X$ be a smooth complex projective surface and let $D$ be a transverse arrangement of $n\geq 2$ rational curves. What shall we say about $E(Y,\widetilde{D})$ of the associated surface $(Y,\widetilde{D})$? In particular, is it true that $E(Y,\widetilde{D}) \leq 8/3$?
\end{problem}
\begin{remark}
If $(X,D)$ is a pair such that $D$ is a transversal arrangement with only ordinary double points as singularities, then we will use the convention that $(Y,\widetilde{D})$ coincides with $(X,D)$.
\end{remark}
At first glance, the number $8/3$, which should be an upper bound for the log Chern slopes, seems to appear out of nowhere. However, as we will see in the course of our survey, this number is indeed the upper bound in many cases. Thus, this folkloric expectation has strong empirical support.

\section{Line arrangements and their characteristic numbers}\label{sec:line_arr}
In this section we present a detailed outline on the characteristic numbers of line arrangements, or just log Chern slopes of log-surfaces, associated with pairs of the form $(\mathbb{P}^{2},\mathcal{L})$, where $\mathcal{L} \subset \mathbb{P}^{2}$ is a line arrangement. It is obvious that a line arrangement in $\mathbb{P}^{2}$ are transverse arrangement and we are able to use all machinery introduced previously. Let us start with a short historical overview. One of the first papers devoted to log-surfaces associated with line arrangements is Iitaka's paper \cite{Iitaka}. As it turned out some years later, the main result of this paper is not correct, and we will explain this error in detail later. Then log-surfaces appeared in the seminal paper of Hirzebruch \cite{Hirzebruch}, but not in a direct way, where we can find a beautiful construction of ball quotient surfaces, i.e., minimal complex projective surfaces of general type satisfying the BMY--inequality. The main ingredient of Hirzebruch's construction is the full abelian cover of $\mathbb{P}^{2}$ branched at line arrangements. Let us recall this construction, since it will be very useful in the course of this section.

Let $\mathcal{L} = \{\ell_{1}, \ldots, \ell_{d}\} \subset \mathbb{P}^{2}$ be an arrangement of $d\geq 6$ lines\footnote{This technical assumption is used directly in Hirzebruch's proof to demonstrate the effectiveness of a specific divisor \cite[p. 127]{Hirzebruch}.}. We assume here that $t_{d}(\mathcal{L}) = 0$, which means that our arrangement is not a pencil of lines. Let us denote by $s_{i}$ the defining linear forms of $\ell_{i}$, i.e., $\ell_{i}$ is the zero locus of $s_{i}$. Consider the following map 
\[f : \mathbb{P}^{2} \ni x \mapsto (s_{1}(x): \ldots : s_{d}(x)) \in \mathbb{P}^{d-1}.\]
Obviously $f$ is well-defined since, by the assumption that $t_{d}(\mathcal{L})=0$, at every point $x$ at least one of the $s_{j}(x)$'s is non-zero.
Now we use the standard Kummer covering
\[\textrm{Km}_{n}: \mathbb{P}^{d-1}\ni (y_{1}: \ldots :y_{d}) \mapsto (y_{1}^{n}: \ldots : y_{d}^{n}) \in \mathbb{P}^{d-1},\]
where the integer $n\geq 2$ is called \textit{the exponent} of $\textrm{Km}_{n}$. This covering is of degree $n^{d-1}$ with the Galois group $(\mathbb{Z}/n\mathbb{Z})^{d-1}$, and this covering is branched along $y_{1} \cdot \ldots \cdot y_{d} = 0$. Our main object of interest is the following fiber product:
\begin{equation}
\label{XN}
X_{n} := \mathbb{P}^{2} \times_{\mathbb{P}^{d-1}} \mathbb{P}^{d-1}= \{(x,y) \in \mathbb{P}^{2} \times \mathbb{P}^{d-1} : f(x) = \textrm{Km}_{n}(y)\}.
\end{equation}
Since there exists a projective transformation on $\mathbb{P}^{2}$ such that $\ell_{1} = \{x_{1}=0\}$, $\ell_{2} = \{x_{2} = 0\}$, and $\ell_{3} = \{x_{3} = 0 \}$, we can study $X_{n}$ by looking at this surface as an embedded object in $\mathbb{P}^{d-1}$. Indeed, $X_{n}$ is defined by equations
\[
X_{n} = \{ (y_{1}: \ldots :y_{d})\in \mathbb{P}^{d-1} : y_{j}^{n} = s_{j}(y_{1}^{n}, y_{2}^{n}, y_{3}^{n}) \,\, \textrm{for} \,\, j \in \{4, \ldots ,d\} \}.
\]
Our surface $X_{n}$ is given by $(d-3)$ homogeneous equations in $\mathbb{P}^{d-1}$, which means that $X_{n}$ is a complete intersection. As Hirzebruch's show in \cite{Hirzebruch}, $X_{n}$ is singular over a point $p$ of the arrangement $\mathcal{L}$ if and only if $p$ is a point of multiplicity $\geq 3$. Let $Y_{n}$ be desingularization of $X_{n}$. Since $Y_{n}$ is a smooth complex projective surface, we can compute its Chern numbers: 
\[K^{2}_{Y_{n}}/n^{d-3}= n^{2}(9-5d + 3f_{1} - 4f_{0}) + 4n(d-f_{1}+f_{0})+f_{1}-f_{0}+d+t_{2},\]
\[e(Y_{n})/n^{d-3} = n^{2}(3-2d+f_{1}-f_{0}) + 2n(d-f_{1}+f_{0}) + f_{1}-t_{2},\]
where $f_{0} := \sum_{r\geq 2} t_{r}$ and $f_{1}:=\sum_{r\geq 2} rt_{r}$, and we refer to \cite[pp. 123 -- 125]{Hirzebruch} or \cite[Chapter 24]{CCC} to see how these computations are performed. 
We are in the middle of the discussion, so in order to give the reader more feeling, we should discuss some additional steps contained in Hirzebruch's work. As we have mentioned, Hirzebruch wanted to construct new examples of ball-quotient surfaces, and in order to do so, he applied his construction to various classically known examples of line arrangements and different exponents. By very detailed and technical considerations presented in \cite{BHH87}, the surface $Y_{n}$ can be a ball-quotient only if $n \in \{2,3,5\}$, which is a very deep observation. Next, not every line arrangement leads us to the situation where $Y_{n}$ is a surface of general type or with non-negative Kodaira dimension, so Hirzebruch has found some mild necessary conditions for line arrangements that guarantee that $Y_{n}$ is a surface with non-negative Kodaira dimension. More precisely, for $n=2$, one must have $t_{d}(\mathcal{L}) = t_{d-1}(\mathcal{L}) = t_{d-2}(\mathcal{L}) = 0$. For $n\geq 3$, it is sufficient that $t_{d}(\mathcal{L}) = t_{d-1}(\mathcal{L}) = 0$. After this digression, let us come back to the construction. Assuming that $Y_{n}$ has non-negative Kodaira dimension, we can use the BMY--inequality \cite{Miyaoka77} and one has $K^{2}_{Y_{n}} \leq 3e(Y_{n})$. This implies that that the following Hirzebruch polynomial
\[H_{\mathcal{L}}(n) = \frac{3e(Y_{n}) - K^{2}_{Y_{n}}}{n^{d-3}} = n^{2}(f_{0}-d) + 2n(d-f_{1}+f_{0}) + 2f_{1}+f_{0}-d-4t_{2}\]
is non-negative for every $n\geq 2$.

Consider the case $n=3$. Recall that for $n=3$ our surface $Y_{n}$ has non-negative Kodaira dimension if $t_{d}(\mathcal{L}) = t_{d-1}(\mathcal{L})=0$ and $d\geq 6$. 
Evaluating $H_{\mathcal{L}}$ at $n=3$, we get 
\begin{equation}
\label{ball}
t_{2}+t_{3} \geq d +\sum_{r\geq 5}(r-4)t_{r}.
\end{equation}
This inequality is known as \textbf{Hirzebruch's inequality for complex line arrangements in the plane}.
Now if we find a line arrangement $\mathcal{L} \subset \mathbb{P}^{2}$ with $d\geq 6$ lines and $t_{d}(\mathcal{L}) = t_{d-1}(\mathcal{L}) = 0$ satisfying equality in \eqref{ball}, then the resulting surface $Y_{3}$ will be a ball-quotient.
\begin{example}
The Hesse arrangement $\mathcal{H}$ consisting of $d=12$ lines with $t_{2}(\mathcal{H})=12$ and $t_{4}(\mathcal{H})=9$ satisfies \eqref{ball}, hence the associated surface $Y_{3}$ is a ball-quotient surface. 
\end{example}
It is worth emphasizing that for $n=3$ the Hesse arrangement is the unique line arrangement satisfying equality in \eqref{ball}, which is an astonishing result, and we refer to \cite[Kapitel 3.1 G.]{BHH87} for details.
\begin{remark}
Although this was not initially expected, Hirzebruch's inequality \eqref{ball} has proven to be a powerful tool in combinatorics. First of all, Hirzebruch's inequality implies that every arrangement of $d \geq  6$ lines such that $t_{d}=t_{d-1}=0$ contains double or triple points as the intersections. This is an important and a non-trivial observation from the perspective of the classical \textit{orchard problem}. Many combinatorialists working on extreme point-line problems, such as the Weak Dirac Conjecture, have also used this inequality, see for instance \cite{Payne}.
\end{remark}
Now we can pass to the world of log-surfaces. Hirzebruch in \cite{Hirzebruch} defined the characteristic number $\gamma(\mathcal{L})$ of a given line arrangement $\mathcal{L} \subset \mathbb{P}^{2}$ via the Chern numbers of the associated surface $Y_{n}$ by the following limit
\begin{equation}\label{eq:limit_gamma}
\gamma(\mathcal{L}) = \textrm{lim}_{n \rightarrow \infty} \frac{c_{1}^{2}(Y_{n})}{c_{2}(Y_{n})} = \frac{9-5d + 3f_{1} - 4f_{0}}{3-2d + f_{1}-f_{0}}.
\end{equation}
If we compute the log Chern numbers using Proposition \ref{chernn} in the setting of line arrangements in the complex projective plane, then we get
\[\overline{c}_{1}^{2}(Y,\widetilde{\mathcal{L}}) = 9 - 5d + 3f_{1} - 4f_{0},\]
\[\overline{c}_{2}(Y,\widetilde{\mathcal{L}}) = 3 - 2d + f_{1} -f_{0},\]
and this allows us to conclude that 
\[E(Y, \widetilde{\mathcal{L}}) = \gamma(\mathcal{L}).\] 
In other words, the characteristic numbers are nothing else log Chern slopes of the associated log-surface $(Y,\widetilde{\mathcal{L}})$, and for this reason we will follow Hirzebruch's concept and use the term characteristic numbers most of the time.

Let us present the first interesting observation that comes from \cite[p. 135]{Hirzebruch}.
\begin{proposition}
Let $\mathcal{L} \subset \mathbb{P}^{2}$ be an arrangement of $d\geq 6$ lines such that $t_{d}=t_{d-1}=t_{d-2}=0$. Then  
\[\gamma(\mathcal{L})  = \frac{5}{2} - \frac{3f_{0}-f_{1}-3}{2(3-2k+f_{1}-f_{0})}. \]
\end{proposition}
\noindent It should be emphasized that Hirzebruch's entire proof boils down to showing that $3-2d+f_{1} - f_{0} > 0$, which does not seem to be easy at first glance. On the other hand, we can try to look at the log Chern numbers of the pair $(Y,\widetilde{\mathcal{L}})$ purely combinatorially, as it was done in \cite{EFU}, where the authors prove the following result.
\begin{proposition}
If $\mathcal{L}$ is an arrangement of $d\geq 4$ lines such that $t_{d}=t_{d-1}=0$, then
the log Chern numbers $\overline{c}_{1}^{2}(Y,\widetilde{\mathcal{L}}), \overline{c}_{2}(Y,\widetilde{\mathcal{L}})$ are positive.
\end{proposition}
\begin{proof}
We reproduce the argument presented in \cite[Proposition 3.3]{EFU}. We will proceed by induction.
If $d=4$, then $\mathcal{L}$ has only nodes as singularities, and we have $\overline{c}_{1}^{2}(Y,\widetilde{\mathcal{L}})=\overline{c}_{2}(Y,\widetilde{\mathcal{L}})=1$. Assume now that $\mathcal{L}$ has $d+1 \geq 5$ lines and let $L \in \mathcal{L}$ be a line passing through $t\geq 3$ points (observe that such a line always exists). Consider now $\mathcal{L}\setminus L$, which is non-trivial, and we can compute the log Chern numbers, namely
\[\overline{c}_{1}^{2}(Y,\widetilde{\mathcal{L}}) \geq \overline{c}_{1}^{2}(Y,\widetilde{\mathcal{L}} \setminus \widetilde{L})  -5 + 2t \geq \overline{c}_{1}^{2}(Y, \widetilde{\mathcal{L}}\setminus \widetilde{L} ) + 1 \geq 1,\]
and
\[\overline{c}_{2}(Y,\widetilde{\mathcal{L}}) = \overline{c}_{2}(Y, \widetilde{\mathcal{L}} \setminus \widetilde{L}) - 2 + t \geq \overline{c}_{2}(Y, \widetilde{\mathcal{L}}\setminus \widetilde{L}) + 1 \geq 1.\]
\end{proof}
Let us now present Hirzebruch's result devoted to the characteristic numbers of real line arrangements, i.e., these are line arrangements in the complex projective plane that can be defined over the real numbers via the canonical embedding $\mathbb{R} \subset \mathbb{C}$.
\begin{theorem}
\label{charR}
 Let $\mathcal{L}\subset \mathbb{P}^{2}$ be a real arrangement of $d\geq 4$ lines such that $t_{d}=t_{d-1}=0$, then
$\gamma(\mathcal{L}) \leq \frac{5}{2}$, and equality holds if and only if $\mathcal{L}$ is simplicial, i.e., every connected component of the complement $\mathbb{P}^{2}_{\mathbb{R}}\setminus \bigcup_{\ell \in \mathcal{L}}\ell$ is an open $2$-dimensional simplex.    
\end{theorem}

We will need the following fundamental result due to Melchior \cite{Melchior}.
\begin{theorem}[Melchior]
Let $\mathcal{L} \subset \mathbb{P}^{2}_{\mathbb{R}}$ be a real arrangement of $d\geq 3$ lines such that $t_{d}=0$. Then
\[t_{2} \geq 3 + \sum_{r\geq 4}(r-3)t_{r},\]
and equality holds if and only if $\mathcal{L}$ is simplicial.
\end{theorem}
Now we can proceed with our proof of Theorem \ref{charR}.
\begin{proof}
Assume that $\gamma(\mathcal{L})>\frac{5}{2}$, then
\[\gamma(\mathcal{L})=\dfrac{9-5d+\sum_{r\geq 2}(3r-4)t_{r}}{3-2d+\sum_{r\geq 2}(r-1)t_{r}}>\dfrac{5}{2},
\]
or equivalently
\[2(9-5d+\sum_{r\geq 2}(3r-4)t_{r})>5(3-2d+\sum_{r\geq 2}(r-1)t_{r}).\]
After some simple manipulations, we obtain
\[3+ \sum_{r\geq 4}(r-3)t_{r}>t_{2},\]
which contradicts Melchior's inequality mentioned above. Next, note that the condition $\gamma(\mathcal{L}) = \frac{5}{2}$ is equivalent to $t_{2} = 3 + \sum_{r\geq 3}(r-3)t_{r}$, which means that $\mathcal{L}$ must be simplicial by Melchior's theorem.
\end{proof}
As we can see, simplical arrangements play an important role in the geography problem, but they are also important research objects in algebraic combinatorics \cite{Grunbaum}. Somehow surprisingly, simplical arrangements of lines are not well understood because they have not yet been fully classified, and this is one of the most fundamental structure problems in this theory, see \cite{Cuntz}. We know that there are three infinite series of simplicial line arrangements, but there is one particular family that we can use to conclude the following.
\begin{proposition}
There exists infinitely many real line arrangements $\mathcal{L}\subset \mathbb{P}^{2}$ such that $\gamma(\mathcal{L}) = \frac{5}{2}$.  
\end{proposition}
\begin{proof}
We need to construct an infinite series of simplicial line arrangements with $t_{d}=t_{d-1}=0$. Let us describe a polyhedral family $\mathcal{P}_{k}$ with $k\geq 3$, namely we consider an arrangement of $2k$ lines where $k$ of the lines are the sides of a regular $k$-gon in the plane, and other $k$ lines are the lines of bilateral symmetry of the $k$-gon, i.e., angle bisectors and perpendicular bisectors of the sides. Observe that $t_{k}=1$ (this is the centre of our $k$-gon), $t_{2}=k$ (these are the midpoints of the sides), and finally $t_{3} = \binom{k}{2}$ (these are the intersections of pairs of sides with the line of symmetry between these sides). It is easy to check that for every $k\geq 3$ arrangement $\mathcal{P}_{k}$ is simplicial, since we obtain equality in Melchior's result, and this implies that $\gamma(\mathcal{P}_{k})=\frac{5}{2}$, which completes the proof.
\end{proof}
\begin{figure}[h]
    \centering
    \begin{tikzpicture}[line cap=round,line join=round,x=2.0cm,y=2.0cm,scale=0.6]
\clip(-2.1097866603487194,-1.653492674462135) rectangle (2.2131214013307456,1.8584890512464292);
\draw [line width=.5pt] (0.,-1.653492674462135) -- (0.,1.8584890512464292);
\draw [line width=.5pt,domain=-2.1097866603487194:2.2131214013307456] plot(\x,{(-0.5773502691896257-0.*\x)/1.});
\draw [line width=.5pt,domain=-2.1097866603487194:2.2131214013307456] plot(\x,{(-0.-1.7320508075688772*\x)/3.});
\draw [line width=.5pt,domain=-2.1097866603487194:2.2131214013307456] plot(\x,{(--3.4641016151377544-5.196152422706632*\x)/3.});
\draw [line width=.5pt,domain=-2.1097866603487194:2.2131214013307456] plot(\x,{(-3.4641016151377544-5.196152422706632*\x)/-3.});
\draw [line width=.5pt,domain=-2.1097866603487194:2.2131214013307456] plot(\x,{(-0.--1.7320508075688772*\x)/3.});
\end{tikzpicture}
\caption{The polyhedral arrangement $\mathcal{P}_3$: 6 lines given by the sides of a regular triangle and its three lines of symmetry (medians). It is simplicial; the only singularities are 3 double points (the side midpoints) and 4 triple points (the three vertices and the center), so $(d,t_2,t_3)=(6,3,4)$ and $\gamma(P_3)=5/2$ (the equality case for real arrangements in Theorem \ref{charR}).}
\end{figure}
What about complex line arrangements? The most important result in this direction was proved by Sommese \cite{Sommese}.
\begin{theorem}
\label{somm}
Let $\mathcal{L} \subset \mathbb{P}^{2}$ be an arrangement of $d\geq 6$ lines such that $t_{d}=t_{d-1}=0$, then one has 
\[\gamma(\mathcal{L}) \leq \frac{8}{3},\]
and equality holds if and only if $\mathcal{L}$ is the dual Hesse arrangement consisting of $d=9$ lines and $t_{3}=12$, and no other intersections.
\end{theorem}
In order to prove this result, we will need a modified version of Hirzebruch's inequality \cite[p. 140]{Hirzebruch}.
\begin{theorem}[Stronger Hirzebruch's inequality]
Let $\mathcal{L} \subset \mathbb{P}^{2}$ be an arrangement of $d\geq 6$ lines such that $t_{d}=t_{d-1}=0$, then one has
\begin{equation}
t_{2} + \frac{3}{4}t_{3} \geq d + \sum_{r\geq 5}(r-4)t_{r}.
\end{equation}
\end{theorem}
Now we are ready to deliver our proof of Theorem \ref{somm}.
\begin{proof}
Assume that  $\gamma(\mathcal{L})\geq\frac{8}{3}$. 
This means that we have
\[3(9-5d+\sum_{r\geq 2}(3r-4)t_{r}) \geq 8(3-2d+\sum_{m\geq 2}(r-1)t_{r}),\]
or equivalently
\[\sum_{r\geq 5}(r-4)t_{r} + d + 3\geq 2t_{2}+t_{3}\]
Now we are going to use Stronger Hirzebruch's inequality, namely
\[\sum_{r\geq 5}(r-4)t_{r} + d + 3 \geq 2t_{2} - \frac{4}{3} t_{2} + \dfrac{4}{3}\left(\sum_{r\geq 5}(r-4) t_{r} +d  \right)\]
which finally gives us
\begin{equation}\label{2.1}
\sum_{r\geq 5}(r-4)t_{r} + d + 2t_{2} \leq 9.    
\end{equation}
In particular, we see see that $d\leq 9$. Now our proof proceed by combinatorial considerations depending on the number of lines. 

For $d=6$ we can list all possible line arrangements in complex plane by hand, or we can use a shortcut and check the database \cite{Mat} to conclude that there is no line arrangement $\mathcal{L}$ with $d=6$ and $t_{d}=t_{d-1}=0$, satisfying \eqref{2.1}, and such that $\gamma(\mathcal{L}) \geq \frac{8}{3}$.

For $d=7$, let us observe that $\eqref{2.1}$ gives us that $2t_{2} + t_{5} \leq 2$, since $t_{7}$ and $t_{6}$ must both equal to zero, and therefore either $t_{2}=0$ and $t_{5}\leq 2$, or $t_{2}=1$ and $t_{5}=0$, so we have two cases to consider. Recall that we have the following naive combinatorial count:
\begin{equation}
\binom{d}{2} = \sum_{r\geq 2}\binom{r}{2}t_{r}.
\end{equation}
If $t_{2} =1$ and $t_{5}=0$, then by Stronger Hirzebruch's inequality we obtain
\[t_{2} + \frac{3}{4} t_{3} = 1 + \frac{3}{4} t_{3} \geq d = 7,\]
so we obtain $t_{3} \geq 8$. If we now plug this data into the naive combinatorial count, we get
\[\binom{7}{2} = 21 = t_{2} + 3t_{3} + 6t_{4} \geq 3\cdot 8,\]
a contradiction.

\noindent In the second case, we have $t_{2} = 0$ and $t_{5}\leq 2$. By Stronger Hirzebruch's inequality, we get
\[t_{3} \geq \frac{4}{3}d + \frac{4}{3}t_{5} \geq \frac{4}{3}d = \frac{28}{3}.\]
As $t_{3}$ is a non-negative integer, we can conclude that $t_{3}\geq 10$.
Plugging this data into the naive combinatorial count, we get
\[\binom{7}{2} = 21 = 3t_{3} + 6t_{4} + 10 t_{5}\geq 3\cdot 10,\]
a contradiction.

For $d=8$, let us observe that $\eqref{2.1}$ gives us that $2t_{2} + t_{5} + 2t_{6} \leq 1$, since $t_{8}$ and $t_{7}$ must both equal to zero, and therefore $t_{2}=t_{6}=0$ and $t_{5}\leq 1$.  By Stronger Hirzebruch's inequality, we get
\[t_{3} \geq \frac{4}{3}d + \frac{4}{3}t_{5} \geq \frac{4}{3}d=\frac{32}{3},\]
and since $t_{3}$ is a non-negative integer, we can conclude that $t_{3}\geq 11$.
Plugging this data into the naive combinatorial count, we get
\[\binom{8}{2} = 28 = 3t_{3} + 6t_{4} + 10 t_{5} \geq 3\cdot 11,\]
a contradiction.

Finally, let us consider the case $d=9$. By $\eqref{2.1}$ and using the assumption that $t_{9}=t_{8}=0$, we get $2t_{2} + t_{5} + 2t_{6}+3t_{7}\leq 0$, and hence we must have $t_{2}=t_{5}=t_{6}=t_{7}=t_{8}=t_{9}=0$. Stronger Hirzebruch's inequality implies in this situation that $t_{3} \geq 12.$
Plugging this data into the naive combinatorial count, we get
\[\binom{9}{2} = 36 = 3t_{3} + 6t_{4} \geq 3\cdot 12,\]
and this means that the only possibility is $t_{3}=12$. Summing up, the only admissible case is $d=9$ and $t_{3}=12$, so $\mathcal{L}$ must be the dual Hesse arrangement, since it is unique up to the projective equivalence \cite{LW}.
\end{proof}
\begin{figure}[h]
\includegraphics[scale=1.4]{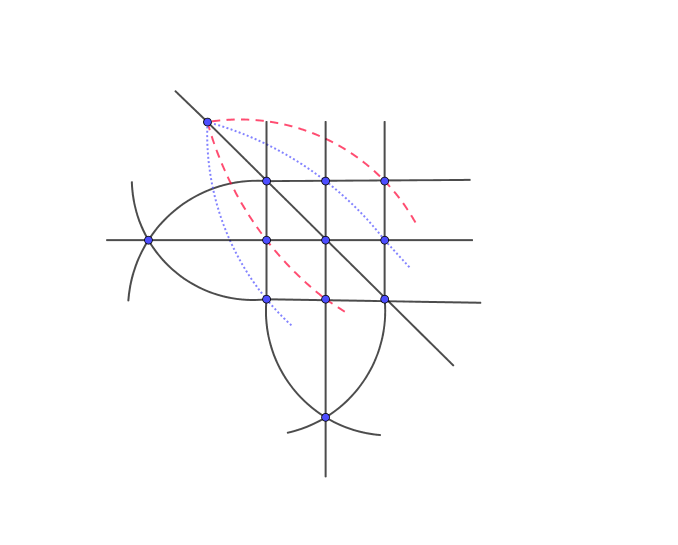}
\caption{The dual Hesse arrangement $\mathcal{L}$ of nine lines in $\mathbb{P}^2$ (incidence type $(9_4,12_3)$): all $12$ intersection points are triple (no double points), each line contains four triple points, and $\mathcal{L}$ is unique up to projective equivalence; it realizes the equality case $\gamma(\mathcal{L})=8/3$ in Theorem \ref{somm}.}\label{fig:dual_Hesse}
\end{figure}

\begin{example}\label{ex:limit_eight_three}
Fix the dual Hesse arrangement $\mathcal{L}$ of 9 lines with only triple intersections (exactly 12), cf. Figure \ref{fig:dual_Hesse}. For the desingularized Kummer covers $Y_n$ associated to $\mathcal{L}$ \eqref{XN}, the Chern number formulas \ref{eq:limit_gamma} give
\begin{equation}
\frac{c_1^2(Y_n)}{c_2(Y_n)} = \frac{24n^2-60n+33}{9n^2-30n+36}\xrightarrow[n\to\infty]{}\frac{24}{9}=\frac{8}{3}.
\end{equation}

Thus $\frac{8}{3}$ occurs as a \textit{limit} of Chern slopes of smooth surfaces $Y_n$. 
\end{example}

\begin{remark}
In \cite{Iitaka}, Iitaka claimed falsely that for complex line arrangements one always has $\gamma(\mathcal{L}) \leq \frac{5}{2}$, and this mistake was already indicated by Hirzebruch in \cite{Hirzebruch}.
\end{remark}
\begin{remark}
The original proof of Sommese contains a computational mistake at the very beginning of his proof, namely claims that the condition $\gamma(\mathcal{L}) \geq \frac{8}{3}$ implies that
\[\sum_{m\geq 5}(m-4)t_{m} + d + 3\geq \textbf{3}\cdot t_{2}+t_{3}\]
but in fact we have
\[\sum_{m\geq 5}(m-4)t_{m} + d + 3\geq 2t_{2}+t_{3}.\]
This seems to be a small error, but it is rather significant due to all combinatorial considerations preformed within the proof. This error was already gently pointed out in \cite[Kapitel 3.4 I]{BHH87}, where the authors gave a modified proof, compared to Sommese's, of the above result. Our proof presented above can be placed in between, since we use Sommese's ideas with necessary modifications.
\end{remark}
To conclude this section, we would like to elaborate a bit on a particular research problem and some predictions regarding the possible values of characteristic numbers for line arrangements.

The most comprehensive result devoted to the geography of characteristic numbers of real line arrangements is the following \cite{EFU}.
\begin{proposition}
The set of accumulations points of characteristic numbers of real line arrangements is the interval $[2, \frac{5}{2}]$.    
\end{proposition}
However, what shall we say in the case of complex line arrangements?
\begin{example}
Let us consider the $n$-th CEVA arrangement $\mathcal{L}_{n} \subset \mathbb{P}^{2}$ of $\, 3n \, $ lines with $t_{3}(\mathcal{L}_{n})=n^{2}$ and $t_{n}(\mathcal{L}_{n})=3$. In particular, for $n=3$ arrangement $\mathcal{L}_{3}$ is exactly the dual Hesse arrangement of $9$ lines with $t_{3}=12$. Then
\[\gamma(\mathcal{L}_{n}) = \frac{5n^2-6n-3}{2n^{2}-3n}.\]
It is easy to see that the sequence composed with the characteristic numbers associated with $\mathcal{L}_{n}$ is decreasing and $\textrm{lim}_{n\rightarrow \infty} \gamma(\mathcal{L}_{n}) = \frac{5}{2}$.
\end{example}
Consider an interesting complex line arrangement with no double intersection points. Such arrangements are very rare, and we believe they must be related to complex reflection groups \cite[Open Problem 1.1.6.]{HarB}.
\begin{example}
Let us consider the Klein arrangement $\mathcal{K}$ of $21$ lines with $t_{3}(\mathcal{K})=28$ and $t_{4}(\mathcal{K})=21$. Then
\[\gamma(\mathcal{K}) = \frac{53}{20}.\]
Observe that $\gamma(\mathcal{K})$ is very close to the theoretical upper-bound.
\end{example}
We must mention that the CEVA family of arrangements is the only known sequence of line arrangements for which the characteristic numbers fall within the interval $(\frac{5}{2}, \frac{8}{3}]$. Based on this observation, we propose the following extremely challenging conjecture:
\begin{conjecture}
For a given $\varepsilon >0$ there exists only finitely many line arrangements $\mathcal{L}$ in the complex projective plane such that $\gamma(\mathcal{L}) > 2.5 + \varepsilon$.
\end{conjecture}

\section{Conic-line arrangements and their characteristic numbers}
In this section, we will study arrangements of conics and lines in the complex plane that only have ordinary singularities. This is the first generalization from line arrangements in the complex plane. However, from the perspectives of topology and geometry, this class of curves behaves completely differently, see for instance \cite{PSz} as an introduction. First, we will present our setup to prepare for the discussion.

Let $\mathcal{CL}\subset \mathbb{P}^{2}$ be an arrangement consisting of $d$ lines and $k$ smooth conics. We assume that $\mathcal{CL}$ is transverse, i.e., all the intersection points are ordinary. Obviously arrangements of conics and lines admit complicated non-ordinary singularities, but here, due to our setting revolved around the geography problem, we focus only on the ordinary singularities. First of all, for our conic-line arrangements in the complex plane the following naive combinatorial count holds:
\begin{equation}
4\cdot\binom{k}{2} + 2kd + \binom{d}{2} = \sum_{r\geq 2} \binom{r}{2}t_{r}.
\end{equation}
Moreover, we have the following Hirzebruch-type inequality recently proved by Pardini and Pokora in \cite{RPPP}.
\begin{theorem}
\label{HirCL}
Let $\mathcal{CL} \subset \mathbb{P}^{2}$ be an arrangement of $d$ lines and $k\geq 3$ smooth conics admitting only ordinary singularities and such that $t_{k+d}=0$. Then one has
\begin{equation}
5k + t_{2} + t_{3} \geq d + \sum_{r\geq 5}(r-4)t_{r}.
\end{equation}
\end{theorem}
Notice that in the above result we do not assume that $d$ is non-negative and, in particular, this result holds for conic arrangements in the plane with ordinary singularities.

In order to get Hirzebruch-type inequality for conic-line arrangements, one can use Hirzebruch's ideas, i.e., we can construct an abelian cover of the complex plane branched along conic-line arrangements with ordinary singularities. It is worth mentioning that the existence of such a covering is not straightforward, and in order to do so we can use a general result by Pardini \cite{RP}.

In the setting of conic-line arrangements in the plane, we have the following log Chern numbers
\begin{align*}
\overline{c}_{1}^{2}(Y, \widetilde{\mathcal{CL}})&= 9 - 8k - 5d + 3f_{1} - 4f_{0},\\    
\overline{c}_{2}(Y, \widetilde{\mathcal{CL}}) &= 3 -2k - 2d + f_{1} - f_{0},
\end{align*}
where $\widetilde{\mathcal{CL}}$ denotes the reduced total transform of $\mathcal{CL}$ under the blowing up of the complex plane along points of multiplicity $\geq 3$. In the same Hirzebruch's spirit, we define the characteristic number of $\mathcal{CL}$ as
\begin{equation}
\gamma(\mathcal{CL}) = \dfrac{9 - 8k - 5d + 3f_{1} -4f_{0}}{3 -2k - 2d + f_{1} - f_{0}}.
\end{equation}
The first important result is devoted to the second log Chern number of $\mathcal{CL}$.
\begin{proposition}[{\cite[Proposition 3.8]{RPPP}}]
\label{prop: c2}
Let $\mathcal{CL} = \{\ell_{1}, \ldots, \ell_{d}, C_{1}, \ldots, C_{k}\} \subset \mathbb{P}^{2}_{\mathbb{C}}$ be an arrangement of $d$ lines and $k\geq 3$ smooth conics admitting only ordinary singularities and such that $t_{d+k}=0$. 
Then
\[3-2k-2d+\sum_{r\geq 2}(r-1)t_{r}>0.\] 
\end{proposition}
Now we are ready to prove our result in this section.
\begin{theorem}
\label{charn}
Let $\mathcal{CL} \subset \mathbb{P}^{2}$ be an arrangement of $d$ lines and $k\geq 3$ smooth conics such that $t_{k+d}=0$. Then $\gamma(\mathcal{CL}) < \frac{8}{3}$.
\end{theorem}
\begin{proof}
 Assume by contradiction that $\gamma(\mathcal{CL}) \geq \frac{8}{3}$, namely 
\[\frac{9-8k-5d + \sum_{r\geq 2}(3r-4)t_{r}}{3-2d-2k + \sum_{r\geq 2}(r-1)t_{r}}\geq \frac{8}{3}.\]
By  Proposition \ref{prop: c2}, the denominator $3-2d-2k+ \sum_{r\geq 2}(r-1)t_{r}$ is strictly positive, so we get 
\[3+d+f_{1}-4f_{0}\geq 8k.\]
By the Hirzebruch-type inequality from Theorem \ref{HirCL}, we have 
\[5k-t_{2}\geq d+ \sum_{r\geq 2}(r-4)t_{r},\]
which implies that
\[3+5k-t_{2} \geq 3+d+\sum_{r\geq 2}(r-4)t_{r} \geq 8k.\]
This gives us
\[3\geq 3k+t_{2},\]
but $k\geq 3$ and $t_{2}\geq 0$, so we arrive at a contradiction.
\end{proof}
\begin{remark}
Our result positively answers on Urz\'ua's question formulated in his doctoral thesis \cite[Section 2.5.1]{Urzua}, where he asked whether
\[\gamma(\mathcal{CL}) \leq \frac{8}{3}\]
holds for conic-line arrangements in the complex plane with ordinary singularities.
\end{remark}
\begin{corollary}
In the setting of Theorem \ref{HirCL}, we have
\[8k + 2t_{2}+t_{3} > 3+d + \sum_{r\geq 5}(r-4)t_{r}.\]
\end{corollary}
\begin{proof}
It follows from the inequality below (and some computations):
\[\frac{9-8k-5d + \sum_{r\geq 2}(3r-4)t_{r}}{3-2d-2k + \sum_{r\geq 2}(r-1)t_{r}} < \frac{8}{3}.\]
\end{proof}
Now we provide a sharper bound on the characteristic numbers under the assumption that intersection points are only double and triple points.
\begin{proposition}
Let $\mathcal{CL} \subset \mathbb{P}^{2}$ be an arrangement of $d$ lines and $k\geq 3$ smooth conics such that the only intersection points are ordinary double and triple points. Then
\[\gamma(\mathcal{CL}) < \frac{5}{2}.\]
\end{proposition}
\begin{proof}
Recall that
\[\overline{c}_{2}(Y,\widetilde{\mathcal{CL}})=3-2d-2k+t_{2} + 2t_{3}> 0\]
by Proposition \ref{prop: c2}.

In order to prove our statement, assume by contradiction $\gamma(\mathcal{CL}) \geq \frac{5}{2}$. It means that
\[\frac{9-8k-5d+2t_{2}+5t_{3}}{3-2d-2k+t_{2}+2t_{3}}\geq \frac{5}{2}.\]
This gives us
\[18 - 16k - 10d +4t_{2} + 10t_{3} \geq 15 - 10d - 10k + 5t_{2} + 10t_{3},\]
so finally we arrive at
\[3 \geq 6k+t_{2}.\]
Since $k\geq 3$, we get a contradiction.
\end{proof}
Here we present a conic-line arrangement with the highest known value of the characteristic number.
\begin{example}[Klein's arrangement of conics and lines]
In \cite{PR2019}, Pokora and Ro\'e described in detail a conic-line arrangement $\mathcal{CL}$ consisting of $21$ lines and $21$ conics, and it has $42$ double points, $252$ triple points, and $189$ quadruple points. We can compute that
\[\gamma(\mathcal{CL}) = \frac{9-8\cdot 21 - 5\cdot 21 + 2\cdot 42 + 5\cdot 252 + 8\cdot 189}{3-2\cdot 21 - 2\cdot 21 + 42 + 2\cdot 252 + 3\cdot 189} \approx 2.512,\] and, to the best of our knowledge, this is the highest known value for slopes in the class of conic-line arrangements.
\end{example}
To complete the picture, let us also consider the case when $d = 0$, resulting in arrangements of smooth conics with ordinary singularities. To a certain extent, we can say that Naruki's work on conic arrangements with nodes and tacnodes as singularities is the first modern paper devoted purely to conic arrangements in the complex plane \cite{Naruki}. Although inspired by Hirzebruch's work on abelian covers and line arrangements, Tang was the first researcher to study questions related to algebraic surfaces constructed using conic arrangements and abelian covers \cite{Tang}. Similar to Hirzebruch's work, Tang introduced the following.
\begin{definition}[{\cite[p. 439]{Tang}}]
Let $\mathcal{C}\subset \mathbb{P}^{2}$ be an arrangement of $k\geq 4$ smooth conics with ordinary singularities such that $t_{k}=0$. The characteristic number of $\mathcal{C}$ is defined as
\[\gamma(\mathcal{C})= 4 - \dfrac{\sum_{r\geq 2}rt_{r}+3}{3-2k + \sum_{r\geq 2}(r-1)t_{r}}.\]
\end{definition}
Using almost the same argument as in the case of conic-line arrangements with ordinary singularities, we can show the following \cite{Pokora1}.
\begin{proposition}
Let $\mathcal{C}\subset \mathbb{P}^{2}$ be an arrangement of $k\geq 4$ smooth conics with only ordinary singularities such that $t_{k}=0$. Then one has
\[\gamma(\mathcal{C}) < \frac{8}{3}.\]
\end{proposition}
To conclude this section, we would like to elaborate on our expectations regarding the characteristic numbers for conic-line arrangements. As with line arrangements in the complex plane, we are unaware of any constructions of families of arrangements of $d$ lines and $k$ conics with $d+k \rightarrow \infty$ such that $\gamma(\mathcal{CL}) > \frac{5}{2}$. This can be explained by the simple fact that constructing symmetric arrangements of curves in the plane is notoriously difficult. Even extreme cases of line arrangements with $\gamma(\mathcal{L}) > 5/2$ are specific due to their underlying properties; these arrangements are related to complex reflection groups. Based on our computational experiments, we can formulate the following challenging conjecture:

\begin{conjecture}
For a given $\varepsilon >0$ there exists only finitely many conic-line arrangements $\mathcal{CL}$ with ordinary singularities in the complex projective plane such that $\gamma(\mathcal{CL}) > 2.5 + \varepsilon$.
\end{conjecture}
 
\section{Log-surfaces constructed as pairs of K3 surfaces and rational curve arrangements}
In this section we will study log-surfaces constructed as pairs consisting of smooth complex projective $K3$ surfaces and transverse arrangements of smooth rational curves. This setting is completely different from our previous case of the complex plane. First of all, smooth rational curves in $K3$ surfaces are rigid, which causes some perturbations. Up to now we were working with ample curves with positive self-intersections, but from now on we are working with curves that may form arrangements of disjoint curves, and their self-intersections are negative. This change is drastic, and our intuition, based on our considerations for arrangements in the complex plane, turns out to be completely misleading. Firstly, since we have to build our intuition  in this new setting, we start with recalling some basic properties of rational curve arrangements in $K3$ surfaces and we present several useful result that we will use in the whole section. Our considerations will be supported by many examples that should help to understand better all the discrepancies compared to the complex plane setting. 

From now on, let $X$ be a smooth complex projective $K3$ (i.e., a smooth complex projective surface having the irregularity $q(X) = 0$ and $K_{X} = \mathcal{O}_{X}$) and let $\mathcal{C}$ be a transverse arrangement of smooth rational curves in $X$. Recall that by a transverse arrangement we mean an arrangement admitting only ordinary singularities, and our arrangement is connected, i.e., we do not allow the case of arrangements containing some disjoint rational curves in $X$. Let $\pi : Y \rightarrow X$ be the blowing up of $X$ along $k\geq 3$-fold points of $\mathcal{C}$, and denote by $\widetilde{\mathcal{C}}$ the reduced total transform of $\mathcal{C}$. The log Chern numbers of the pair $(Y, \widetilde{\mathcal{C}})$ are the following:
\[\overline{c}_{1}^{2}(Y, \widetilde{\mathcal{C}}) = - 2n + 3f_{1}-4f_{0},\]
\[\overline{c}_{2}(Y, \widetilde{\mathcal{C}})= 24 - 2n + f_{1} - f_{0}.\]
Here we will only consider pairs in which both $\overline{c}_1^2(Y, \widetilde{\mathcal{C}})$ and $\overline{c}_2(Y, \widetilde{\mathcal{C}})$ are positive.
\begin{definition}
In the setting as above, we define the log Chern slope of a given pair $(Y, \widetilde{\mathcal{C}})$ consisting of a smooth complex projective $K3$ surface and a transverse arrangement of rational curves as
\[E(Y, \widetilde{\mathcal{C}}) = \frac{- 2n + 3f_{1}-4f_{0}}{24 - 2n + f_{1}-f_{0}}.\]
\end{definition}
In order to approach Problem \ref{main} devoted to the geography of log-surfaces, we are going to use the following result proven by Pokora and Laface in \cite{LP}.
\begin{theorem}
\label{Hir}
Let $X$ be a smooth complex projective $K3$ surface and let $\mathcal{C}$ be a transverse arrangement  of $n\geq 2$ smooth rational curves. Then we have
\begin{equation}
4n - t_{2} + \sum_{r\geq 3}(r-4)t_{r} \leq 72.
\end{equation}
\end{theorem}
Our first goal here is to find some combinatorial conditions for the transverse arrangement of rational curves $\mathcal{C}$ in a smooth complex projective $K3$ surface such that the associated log-surface could have high log Chern slopes. We want to find a numerically accessible condition that is easy to check. To do this, we assume now that $E(Y, \widetilde{\mathcal{C}}) \geq 8/3$.
This condition tells us that
\[ 3\cdot\bigg(-2n + \sum_{r\geq 2}(3r-4)t_{r} \bigg) \geq 8 \cdot \bigg(24 - 2n + \sum_{r\geq 2}(r-1)t_{r} \bigg).\]
This leads us to
\begin{equation}
10n + \sum_{r\geq 2}(r-4)t_{r} \geq 192.
\end{equation}
Using Theorem \ref{Hir}, we get
\[72+6n - t_{2}\geq 10n + \sum_{r\geq 2}(r-4)t_{r} \geq 192.\]
Based on that, we set our criterion, namely
\begin{equation}
\label{condition}
n > 20 + \frac{1}{6}t_{2}.
\end{equation}

Now our main goal is to decide whether there exists a pair consisting of a smooth complex projective $K3$ surfaces $X$ and a transverse arrangement of rational curves $\mathcal{C}$ in $X$ such that \eqref{condition} holds and indeed $E(Y, \widetilde{\mathcal{C}}) \geq 8/3$. It turns out that the answer is affirmative and we have some interesting examples.
\begin{proposition}
\label{maint}
There exists a configuration of $32$ lines $\mathcal{C}$ in the Fermat quartic surface $\mathcal{F}_{4} \, : \, x^4 + y^4 + z^4 + w^4 = 0$ in $\mathbb{P}^{3}$ such that
\[E(Y, \widetilde{\mathcal{C}}) = 8/3.\]
\end{proposition}
\begin{proof}
It is classically known that on the smooth Fermat quartic surface there exists a configuration of $48$ lines intersecting at exactly $192$ double and $24$ quadruple intersection points. As we can expect, the mentioned configuration of lines is very symmetric and due to this fact we have further unexpected properties of the whole configurations.

It turns out that the set $\mathcal{C}_{48}$ of all $48$ lines can be divided into three disjoint sets $\mathcal{C}^{1}, \mathcal{C}^{2}, \mathcal{C}^{3}$ of $16$ lines with the property that lines in each set intersect only in $8$ quadruple points and these are the only intersection points. Moreover, these sets of quadruple points are pairwise disjoint and taken together they build the set of all $24$ quadruple intersection points for the set of $48$ lines in $\mathcal{F}_{4}$.

Consider now $C:= C^{i} \cup C^{j}$ with $i\neq j$. This arrangement consists of $32$ lines intersecting at $16$ quadruple intersection points and exactly $64$ double intersection points (and there are no other intersections). Observe that we have
\[32 > 20 + \frac{1}{6}t_{2} = 20 + 10\frac{2}{3}.\]
Now we are going to compute the log Chern slope, namely
\[E(Y, \widetilde{\mathcal{C}})= \frac{-2n + 2t_{2} + 8t_{4}}{24 - 2n + t_{2} + 3t_{4}}= \frac{-64 + 128 + 128}{24-64 + 64+48} = \frac{192}{72}=\frac{8}{3},\]
which completes the proof.
\end{proof}
\begin{remark}
As our proof shows, we have found exactly $3$ different arrangements of $32$ lines having the same weak combinatorics (i.e., the same number of lines and the same number of the intersection points of given multiplicities) with the log Chern slope equal to $8/3$.
\end{remark}
\begin{remark}
If we take the whole configuration $\mathcal{C}_{48}$ of $48$ lines in the Fermat quartic surface $\mathcal{F}_{4}$ having $t_{2}=192$ and $t_{4}=24$, then we 
\[E(Y, \widetilde{\mathcal{C}_{48}}) = \frac{-96 + 384 + 192}{24 - 96 + 192 + 72}= \frac{480}{192}=\frac{5}{2}.\]
\end{remark}
\noindent
Now we construct an example of a $K3$ pair with $8/3 > E(Y, \widetilde{\mathcal{C}}) > 5/2$ using the geometry of the Schur quartic surface.
\begin{proposition}
There exists the unique configuration of $48$ lines $\mathcal{C}$ in the Schur quartic surface $\mathcal{S}_{4}$ such that
\[E(Y , \widetilde{\mathcal{C}}) = 2.56.\]
\end{proposition}
\begin{proof}
Let us recall that on the Schur quartic surface we have exactly $64$ lines. More precisely, there are $16$ lines, called as lines of the first kind, and $48$ lines that are called as lines of the second kind -- see \cite[Definition 3.3.10]{Davide} for the definition of kinds of lines. It can be observed, and this is (maybe) a folkloric result, that the $16$ lines of the first kind intersect only in $8$ quadruple points, but we will focus on the arrangement $\mathcal{C}$ consisting of $48$ lines of the second kind. Using discussion from \cite[Section 4.2]{Davide}, we can observe the following. Let $M$ be one of the $48$ lines of the second kind. Then, there exist exactly $10$ planes containing $M$ such that the residual cubic is reducible. In particular, we have
\begin{itemize}
\item In $4$ planes the residual cubic splits into $3$ lines, two of the second kind, say $M_{1}, M_{2}$, and one of the first kind which we denote by $L_{3}$. Notice that the intersection point of $M_{1}$ and $M_{2}$ lies on $M$, but not on $L_{3}$. 
\item In $6$ planes the residual cubic splits into one irreducible conic and one line of the second kind.
\end{itemize}
Based on that combinatorial description we can deduce that the $48$ lines of the second kind intersect in $144$ double and $64$ triple intersection points. Observe that 
\[48 = n > 20 + \frac{1}{6}t_{2} = 20 + 24 = 44,\]
so we might expect that the log Chern slope is high. Indeed, we have
\[E(Y , \widetilde{\mathcal{C}}) = \frac{-2n + 2t_{2} + 5t_{3}}{24-2n + t_{2} + 2t_{3}} = \frac{64}{25} = 2.56.\]
\end{proof}

Now we are going to present an example that shows that \eqref{condition} seems to be a good condition as a simple-to-check criterion.
\begin{example}[One of the most algebraic $K3$ surface]
We consider the $A_{1}(6)$ arrangement of $6$ lines with $4$ triple, and $3$ double intersection points in the complex projective plane. We blow up the four triple points and by that we obtain a del-Pezzo surface $S$ with exactly ten $(-1)$-curves. These ten $(-1)$-curves form a simple normal crossing divisor $D$ with exactly $15$ double intersection points. After blowing up these $15$ double points, we get a surface $S'$ with $15$ curves that are $(-1)$-curves and $10$ curves that are $(-4)$-curves. Now we take a $2:1$ cover of $S'$ branched along ten $(-4)$-curves and we obtain the $K3$ surface $X$ equipped with the configuration $\mathcal{V}$ of $25$ smooth rational curves and $30$ double intersection points. Therefore we have $n=25$ and $t_{2}=30$, so we have 
\[25 = n =  20 + \frac{1}{6}t_{2}.\]
It is easy to compute log Chern numbers, namely
\[\overline{c}_{1}^{2}(Y,\widetilde{\mathcal{V}}) = -2n + 2t_{2} = -50 + 60 = 10,\]
\[\overline{c}_{2}(Y,\widetilde{\mathcal{V}}) = 24 -2n + t_{2} = 24 - 50 + 30 = 4,\]
which gives us \[E(Y,\widetilde{\mathcal{V}}) = \frac{5}{2}.\]
\end{example}

\subsection{Combinatorial constraints on K3 pairs with log Chern slopes in [2, 3].}\label{sec:comb_constr_K3}
Now we are going to deliver a complete combinatorial characterization of $K3$ pairs $(X,\mathcal{C})$ such that $E(Y,\widetilde{\mathcal{C}}) = 3$.
\begin{theorem}
\label{bl}
Let $X$ be a smooth complex projective $K3$ surface and let $\mathcal{C}$ be a transverse arrangement of $n\geq 2$ rational curves. Then the following conditions are equivalent:
\begin{enumerate}
\item[a)] $E(Y,\widetilde{\mathcal{C}}) = 3$, 
\item[b)] $4n = 72 + t_{2}+t_{3}$ and $t_{r}=0$ for $r\geq 4$.
\end{enumerate}
\end{theorem}
\begin{proof}
Suppose that $E(Y,\widetilde{\mathcal{C}}) = 3$. Then this condition implies
\[-2n + \sum_{r\geq 2}(3r-4)t_{r} = 72 - 6n + \sum_{r\geq 2}(3r-3)t_{r},\]
and we get
\[4n = 72 + \sum_{r\geq 2}t_{r}.\]
Using Theorem \ref{Hir}, we obtain
\[72 + t_{2} + \sum_{r\geq 3} t_{r} - t_{2} + \sum_{r\geq 3}(r-4)t_{r} = 72 + \sum_{r\geq 3}(r-3)t_{r} \leq 72.\]
Based on that, we have 
\[0 \leq \sum_{r\geq 4}(r-3)t_{r} \leq 0,\]
where the left-hand side inequality follows from the fact that $t_{k}\geq 0$ for all $k\geq 2$. This gives us that $t_{r} = 0$ for $r\geq 4$ and we finally arrive at
\[4n = 72 +t_{2}+t_{3},\]
which completes our proof for the implication $a) \Rightarrow b)$.

Suppose now that $\mathcal{C}$ has only ordinary double and triple intersection points and $4n=72+t_{2}+t_{3}$. Then
\[E(Y,\widetilde{\mathcal{C}}) = \frac{-2n + 2t_{2}+5t_{3}}{24 -2n + t_{2}+t_{3}} = \frac{-2n + 2(4n-72)+3t_{3}}{24 -2n + 4n-72 + t_{3}} = \frac{3(t_{3}+2n-48)}{t_{3}+2n-48} = 3.\]
The proof has been completed.
\end{proof}
The above criterion is rather surprising and contra-intuitive with respect to what we have seen in the case of line arrangements and the complex projective plane since in that scenario high log Chern slopes are obtained for arrangements with high multiplicities of intersection points. Let us present the following example justifying our surprise.
\begin{example}[Special lines on the Fermat quartic surface]
As it was already mentioned, on $\mathcal{F}_{4}$ in $\mathbb{P}^{3}_{\mathbb{C}}$ there are $48$ lines and among these $48$ lines we can pick exactly $16$ lines forming an arrangement $\mathcal{C}^{i}$ with only $8$ quadruple points, where $i \in \{1,2,3\}$. Then we have
\[E(Y,\widetilde{\mathcal{C}^{i}}) = \frac{-2n+8t_{4}}{24 - 2n + 3t_{4}} = \frac{32}{16} = 2.\]
The same log Chern slope can be obtained if we take the Schur quartic surface and $16$ lines of the first kind.
\end{example}

Now we focus on arrangements with only double points. In the case of line arrangements in the complex projective plane with only double intersection points we are able to get the log Chern slope asymptotically equal to $2$ if the number of lines $d \rightarrow \infty$. In the case of rational curves and $K3$ surfaces we have a completely different picture.
\begin{proposition}\label{prop:24_lines}
Let $X$ be a smooth complex projective $K3$ surface and let $\mathcal{C}$ be an arrangement which has only double intersection points (i.e., only nodes). Then $E(X,\mathcal{C})=2$ if and only if $n=24$.
\end{proposition}
\begin{proof}
Assume that $n=24$, then
\[E(X,\mathcal{C}) = \frac{-2n+2t_{2}}{24 - 2n + t_{2}} = \frac{-48+2t_{2}}{-24 + t_{2}}=2.\]
If now $E(X,\mathcal{C})=2$, then
\[-2n + 2t_{2} = 2(24 -2n +  t_{2}) = 48 -4n+2t_{2},\]
which gives us $n=24$.
\end{proof}

\begin{example}\label{ex:24_lines}
    Inspired by the quartic \cite[\S 5.2]{Veniani_lines} with configuration $Q_{56}$, we fix the following smooth projective quartic: %($p=2,q=3$)
    \[
    \begin{split}
    X:72 x_0^4+1394 x_1^2 x_0^2-306 x_2^2 x_0^2-656 x_3^2 x_0^2+72 x_1^4+72 x_2^4\\
    +72 x_3^4-656 x_1^2 x_2^2-306 x_1^2 x_3^2+1394 x_2^2 x_3^2=0    
    \end{split} \]
    We check by a direct computation, cf. \cite{auxiliary_files}, that the scheme of lines $S$ is defined over $\mathbb{Q}$ and $S(\overline{\mathbb{Q}}) = S(\mathbb{Q})$, cf. Section \ref{sec:examples}. Moreover, $|S(\mathbb{Q})|=24$, and the lines in this configuration $\mathcal{C}$ (Table \ref{tab:24_lines}) intersect at $96$ intersection points, each point is an intersection of exactly two lines. Hence $t_{2}=96$ and $t_{r}=0$ for $r>2$. It follows that $E(X,\mathcal{C})=2$, proving that Proposition \ref{prop:24_lines} is not an empty statement.
\end{example}

\subsection{Good reduction and reconstruction method of the line‑incidence graph on quartic K3 surfaces}

This section provides the methodological backbone for Section \ref{sec:comb_constr_K3} and Table \ref{tab:quartic_examples}: it formalizes good reduction of the scheme of lines and proves that, away from finitely many primes, reduction preserves the incidence graph of lines on a quartic K3 surface. Consequently the combinatorial counts $(t_r)$ we use to compute log‑Chern numbers agree in characteristic 0 and at our chosen primes. 

Let $X$ be a smooth projective quartic in $\mathbb{P}^{3}$. We denote by $\mathcal{L}$ the set of lines in $X$. For the latter we assume the following:
\begin{itemize}
	\item $X$ is defined over a number field $K$.
	\item The coordinates in the projective space $\mathbb{P}^{3}$ are $x_0,x_1,x_2$ and $x_3$.
	\item We fix once and for all an algebraic closure of $K$ and denote it by $\overline{K}$.
\end{itemize}

We denote by $\mathcal{S}$ a $0$-dimensional scheme defined over $K$ which parametrizes the lines, i.e., there exists a map $\eta: \mathcal{S}(\overline{K})\rightarrow \mathcal{L}$ such that for each point $p\in\mathcal{S}(\overline{K})$ the line $\eta(p)$ is defined over a certain finite extension of $K$ and lies on $X$. The map $\eta$ is injective.

Let $\mathcal{O}_{K}$ denote the ring of integers in the number field $K$. Let $\mathfrak{p}$ be a maximal ideal in $\mathcal{O}_{K}$ such that $\mathbb{Z}\cap\mathfrak{p} = p\mathbb{Z}$ for a certain prime number $p$.

We say that the surface $X$ has good reduction at $\mathfrak{p}$ if the surface $X$ extends to a smooth morphism $\pi:\mathcal{X}\rightarrow \Spec(\mathcal{O}_{K})_{\mathfrak{p}}=\{o,s\}$ over the localization $(\mathcal{O}_{K})_{\mathfrak{p}}$ of $\mathcal{O}_{K}$ at the maximal ideal $\mathfrak{p}$. The generic fibre $\pi^{-1}(o)$ is $X$ and the special fibre $X_{\mathfrak{p}}:=\pi^{-1}(s)$ is the \textbf{reduction modulo } $\mathfrak{p}$ of the surface $X$. The scheme $X_{\mathfrak{p}}$ is smooth over $\mathbb{F}_{\mathfrak{p}}$, the residue class field of the ideal $\mathfrak{p}$. 

For a $0$-dimensional finite scheme $S$ over $K$ we adapt this definition. We say that the scheme $S$ has good reduction at $\mathfrak{p}$ if there exists an integral model $\mathcal{S}\rightarrow\Spec\mathcal{O}_{K}$ such that the induced morphism $\mathcal{S}\rightarrow\Spec(\mathcal{O}_{K})_{\mathfrak{p}}$ is smooth. This is equivalent to the condition that the fibre $\mathcal{S}_{\mathfrak{p}}$ is a reduced scheme. Such a condition implies automatically that the degree of $S$ is the same as the degree of $\mathcal{S}_{\mathfrak{p}}$. 

Given an integral model $\mathcal{S}\rightarrow\mathcal{O}_{K}$ of the scheme $S$ we can represent $\mathcal{S}$ as a closed subscheme of certain affine space $\mathbb{A}^{n}$ of suitable dimension $n$. In the local coordinates $\{\lambda_{i}\}$ we calculate a basis of the ideal $I$, with $\mathcal{O}_{K}$-coefficients such that $V(I)=S$. Suppose that $I$ is generated by the polynomials $F_{i}\in\mathcal{O}_{K}[\lambda_{1},\ldots,\lambda_{n}]$ with $i=1,\ldots N$. We compute the Jacobian matrix $J=\left(\frac{\delta F_{i}}{\delta \lambda_{j}}\right)_{i,j}$. Since $N\geq n$, we compute the ideal $\mathcal{R}$ generated by the $n-1$ minors of $J$ and the polynomials $F_{i}$. The ideal $\mathcal{R}\cap\mathcal{O}_{K} = d\mathcal{O}_{K}$ is called the discriminant ideal of the integral model $\mathcal{S}$. It follows from the construction that if the maximal ideal $\mathfrak{p}\nmid d$, then $\mathcal{S}$ has good reduction at $\mathfrak{p}$, i.e. the scheme $\mathcal{S}_{\mathfrak{p}}$ is reduced.

We denote by $r_{\mathfrak{p}}:X(\overline{K})\rightarrow X_{\mathfrak{p}}(\overline{\mathbb{F}_{\mathfrak{p}}})$ the map which sends a point $x$ to the point $\overline{x}$ which is the reduction modulo $\mathfrak{p}$ of the point $x$, i.e. $\overline{x} = \iota(x)\cap \pi^{-1}(s)$ where $\iota:X\rightarrow \mathcal{X}$ is the natural inclusion of $X$ in $\mathcal{X}$. By abuse of notation we denote by $r_{\mathfrak{p}}$ the induced map $\mathcal{L}\rightarrow\mathcal{L}_{\mathfrak{p}}$ where the set $\mathcal{L}_{\mathfrak{p}}$ consists of the reductions modulo $\mathfrak{p}$ of the lines in $\mathcal{L}$, i.e. each line in $\mathcal{L}_{\mathfrak{p}}$ is defined by the reduction of coefficients via $r_{\mathfrak{p}}$.
To the pair $(S,\eta)$ we attach the following objects:
\begin{itemize}
	\item A finite set $\mathcal{P}(\mathcal{S},\eta)\subset \mathbb{P}^{3}(\overline{K})$ of points of intersection of every pair of lines in $\mathcal{L}$.
	\item A bipartite finite graph $\mathcal{G}(\mathcal{S},\eta) = (\mathcal{V},\mathcal{E})$, where the vertices are the set $\mathcal{V}=\mathcal{P}\cup\mathcal{L}$ and the edge set $\mathcal{E}$ consists of pairs $(p,\ell)$ of $p\in\mathcal{P}$ and $\ell\in\mathcal{L}$ such that $p\in\ell$.
\end{itemize}

We say that the scheme of lines $(\mathcal{S},\eta)$ has good reduction at $\mathfrak{p}$ when there exists a $0$-dimensional scheme $\mathcal{S}_{\mathfrak{p}}$ defined over $\mathbb{F}_{\mathfrak{p}}$ and a map $\eta_{\mathfrak{p}}$ such that $\eta_{\mathfrak{p}}: \mathcal{S}_{\mathfrak{p}}(\overline{\mathbb{F}_{\mathfrak{p}}})\rightarrow \mathcal{L}_{\mathfrak{p}}$ and:
\begin{itemize}
	\item The map $\eta_{\mathfrak{p}}$ is injective. 
	\item We have $r_{\mathfrak{p}}\circ\eta = \eta_{\mathfrak{p}}\circ r_{\mathfrak{p}}$.
	\item The finite graph $\mathcal{G}(\mathcal{S}_{\mathfrak{p}},\eta_{\mathfrak{p}})$ is isomorphic as bipartite graph with the graph $\mathcal{G}(\mathcal{S},\eta)$ where the isomorphism is induced by the reduction map $r_{\mathfrak{p}}$. 
\end{itemize}

\begin{proposition}\label{prop:good_reduction}
Let $X$ be a smooth projective quartic in $\mathbb{P}^{3}$ defined over the number field $K$. There exists a positive integer $D$ such that for each prime number $p\nmid D$ and a maximal ideal $\mathfrak{p}$ of residue characteristic $p$ the scheme of lines $(\mathcal{S},\eta)$ on $X$ has good reduction at~$\mathfrak{p}$.
\end{proposition}
\begin{proof}
For the graph $\mathcal{G}(\mathcal{S},\eta)$ one can attach a $0$-dimensional scheme $\mathcal{T}$ defined over $K$ such that we have an injective function $\theta:\mathcal{T}(\overline{K})\rightarrow \mathcal{M}$ where $\mathcal{M}$ is a finite set which consists of elements $(p,\{\ell_{1},\ell_{2}\})$ where $\{p\} = \ell_{1}\cap\ell_{2}$ and $\ell_{1}\neq \ell_{2}$. The existence of $\mathcal{T}$ follows from the observation that it can be naturally defined as a closed subscheme of $\mathcal{S}\times\mathcal{S}$ via the intersection condition $\{p\} = \ell_{1}\cap\ell_{2}$. One can notice that the graph $\mathcal{G}(\mathcal{S},\eta)$ is completely determined by the set $\mathcal{M}$, i.e. for each element $(p,\{\ell_{1},\ell_{2}\})\in\mathcal{M}$ we obtain two edges $(p,\ell_{1})$ and $(p,\ell_2)$. Vice-versa: for each point $p\in\mathcal{L}$ and any two distinct edges $(p,\ell_{1}), (p,\ell_{2}) \in \mathcal{E}$  we obtain a unique pair $(p,\{\ell_{1},\ell_{2}\})\in\mathcal{M}$ because the point of intersection of two lines is unique.

Since the scheme $\mathcal{T}$ is $0$-dimensional, we compute for its integral model over $\mathcal{O}_{K}$ its discriminant ideal $d\mathcal{O}_{K}$. We denote by $D$ the norm of the element $d$. It follows from the definition of the discriminant that each maximal ideal $\mathfrak{p}$ above $p\nmid D$ admits a good reduction for the scheme $\mathcal{T}$.

It follows from the construction of the pair $(\mathcal{T},\theta)$ that the analogously defined scheme $(\mathcal{T}_{\mathfrak{p}},\theta_{\mathfrak{p}})$ which encodes the graph $\mathcal{G}(\mathcal{S}_{\mathfrak{p}},\eta_{\mathfrak{p}})$ satisfies the condition $r_{\mathfrak{p}}\circ\theta = \theta_{\mathfrak{p}}\circ r_{\mathfrak{p}}$ if and only if the graph $\mathcal{G}(\mathcal{S}_{\mathfrak{p}},\eta_{\mathfrak{p}})$ is isomorphic as bipartite graph with the graph $\mathcal{G}(\mathcal{S},\eta)$. The reduction map $r_{\mathfrak{p}}:\mathcal{T}(\overline{K})\rightarrow \mathcal{T}_{\mathfrak{p}}(\overline{\mathbb{F}_{\mathfrak{p}}})$ is bijective when the residue characteristic $p$ of the maximal ideal $\mathfrak{p}$ does not divide $N$.
\end{proof}

\begin{corollary}[Slope invariance under good reduction]\label{cor:slope-invariance}
Let $X/K$ be a smooth quartic $K3$ surface, and let $(S,\eta)$ denote its scheme of lines with incidence graph $G(S,\eta)$. Suppose $(S,\eta)$ has good reduction at a prime $\mathfrak{p}$. Then the reduction
$G(S_{\mathfrak{p}},\eta_{\mathfrak{p}})$
of the incidence graph is isomorphic to $G(S,\eta)$. In particular, the multiplicity counts $\{t_r\}$ (the numbers of $r$-fold intersection points) coincide for $X$ and $X_{\mathfrak{p}}$, and therefore the log-Chern numbers and the slope computed from these counts agree.
\end{corollary}

\begin{proof}
By the definition of good reduction, the reduction map induces an isomorphism of the bipartite incidence graphs, so the combinatorial data $\{t_r\}$. Hence the formulas in Section \ref{sec:comb_constr_K3} that depend only on~$\{t_r\}$ are preserved.
\end{proof}

For a given quartic $X$ defined over $K$ we construct the scheme $(\mathcal{S},\eta)$ via the method which involves the use of Gr\"obner bases. We are interested in computing the structure of the graph $\mathcal{G}(\mathcal{S},\eta)$, but we avoid the direct computation of the scheme $(\mathcal{T},\theta)$ due to typically very high degree of the finite extension $L/K$ which realizes the equality $\mathcal{S}(\overline{K}) = \mathcal{S}(L)$ and $\mathcal{T}(\overline{K}) = \mathcal{T}(L)$. In practice, to deal with such extensions we would need excessive computer memory which is beyond the scope of any practical computation.

Instead, we deal with the problem by a careful choice of the prime number $p$ which provides a good reduction of $(\mathcal{S},\eta)$. In principle, such a prime can be easily determined from the scheme $(\mathcal{T},\theta)$ via the proof of Proposition \ref{prop:good_reduction}. But this direct approach is again rather intractable for most practical purposes. We compute the prime $p$ via the following conditions:
\begin{itemize}
	\item[(I)]The scheme $\mathcal{S}_{\mathfrak{p}}$ is reduced.
	\item[(II)]We embed the scheme $\mathcal{S}$ in the ambient affine space $\mathbb{A}$. We determine the map $\eta$ from this embedding. This is practically computed with the Groebner bases method.
	\item[(III)]For each pair of points $L_{1},L_{2}$ we compute whether the corresponding subscheme of pairs $(p,\{\ell_{1},\ell_{2}\})$ in $\mathcal{T}$ determined by $\eta$ is reduced under reduction modulo $p$. We can do it without explicit computation of the geometric points.
\end{itemize}

In each example analyzed in this paper we provide the suitable Magma code to construct explicitly the pairs $(\mathcal{S}_{p}, \eta_{\mathfrak{p}})$, $(\mathcal{T}_{p}, \theta_{\mathfrak{p}})$ by providing the data which determines $\mathcal{S}_{p}(\overline{\mathbb{F}_{\mathfrak{p}}})$ and $\mathcal{T}_{p}(\overline{\mathbb{F}_{\mathfrak{p}}})$. This is enough to construct the graph $\mathcal{G}(\mathcal{S}_{\mathfrak{p}}),\eta_{\mathfrak{p}})$. Under the good reduction hypothesis for the given prime number $p$ we obtain the graph $\mathcal{G}(\mathcal{S},\eta)$ via Proposition \ref{prop:good_reduction}. The graph $\mathcal{G}(\mathcal{S},\eta)$ encodes for us the weak combinatorics of lines which is need to compute the slope in each case.

\begin{remark}
One could approach the problem of the reconstruction of the graph $\mathcal{G}(\mathcal{S},\eta)$ from the data $(\mathcal{S},\eta)$ purely in characteristic $0$ by dealing with the homotopy continuation method which allows one to approximate (with a fixed precision) the set of points $\mathcal{S}(\mathbb{C})$. This will require a usage of specialized package to compute the list of complex points $\mathcal{S}(\mathbb{C})$ with given precision (in interval arithmetic). In the next step one would have to compute the approximation of the set $\mathcal{T}(\mathbb{C})$ which can be done via solutions of approximate linear systems. To avoid the additional issues of precision tracking we decided to not pursue this method. 
Perhaps one could combine both methods by dealing with a suitable approach via $p$-adic methods. We do not intend to study this approach further in this text.
\end{remark}

In the auxiliary files attached to this paper we gather all the data for the examples (Table~\ref{tab:quartic_examples}) studied in this paper \cite{auxiliary_files}.

\begin{ack}
Bartosz Naskr\k{e}cki acknowledges the support by Dioscuri program initiated by the Max Planck Society, jointly managed with the National Science Centre (Poland), and mutually funded by the Polish Ministry of Science and Higher Education and the German Federal Ministry of Education and Research.\\
We would like to thank Piotr Achinger for comments of the first draft of the paper. We thank Davide Cesare Veniani for providing us with extensive computational database of quartics and to \L ucja Farnik for sharing with us here picture of the dual Hesse arrangement. Finally, we would like to thank the anonymous referees who provided us with many useful comments that helped us improve the readability of this survey.
\end{ack}

\begin{funding}
Piotr Pokora is supported by the National Science Centre (Poland) Sonata Bis Grant  \textbf{2023/50/E/ST1/00025}. For the purpose of Open Access, the author has applied a CC-BY public copyright license to any Author Accepted Manuscript (AAM) version arising from this submission.
\end{funding}
\newpage
\section{Tables}
    \begin{table}[htb]
    \begin{tabular}{cccc|cccc}
    $h_{0}$ & $h_{1}$ & $h_{2}$ & $h_{3}$ & $\widetilde{h}_{0}$ & $\widetilde{h}_{1}$ & $\widetilde{h}_{2}$ & $\widetilde{h}_{3}$ \\
    \hline\hline
    1 & 0 & 0 & 3 & 0 & 1 & 1/3 & 0\\
1 & 0 & 0 & 3 & 0 & 1 & -1/3 & 0\\
1 & 0 & 4/5 & 3/5 & 0 & 1 & -3/5 & 4/5\\
1 & 0 & 4/5 & 3/5 & 0 & 1 & 3/5 & -4/5\\
1 & 0 & -4/5 & 3/5 & 0 & 1 & 3/5 & 4/5\\
1 & 0 & -4/5 & 3/5 & 0 & 1 & -3/5 & -4/5\\
1 & 0 & 0 & 1/3 & 0 & 1 & 3 & 0\\
1 & 0 & 0 & 1/3 & 0 & 1 & -3 & 0\\
1 & 0 & 2 & 0 & 0 & 1 & 0 & 1/2\\
1 & 0 & 2 & 0 & 0 & 1 & 0 & -1/2\\
1 & 0 & 1/2 & 0 & 0 & 1 & 0 & 2\\
1 & 0 & 1/2 & 0 & 0 & 1 & 0 & -2\\
1 & 0 & -1/2 & 0 & 0 & 1 & 0 & 2\\
1 & 0 & -1/2 & 0 & 0 & 1 & 0 & -2\\
1 & 0 & -2 & 0 & 0 & 1 & 0 & 1/2\\
1 & 0 & -2 & 0 & 0 & 1 & 0 & -1/2\\
1 & 0 & 0 & -1/3 & 0 & 1 & 3 & 0\\
1 & 0 & 0 & -1/3 & 0 & 1 & -3 & 0\\
1 & 0 & 4/5 & -3/5 & 0 & 1 & 3/5 & 4/5\\
1 & 0 & 4/5 & -3/5 & 0 & 1 & -3/5 & -4/5\\
1 & 0 & -4/5 & -3/5 & 0 & 1 & -3/5 & 4/5\\
1 & 0 & -4/5 & -3/5 & 0 & 1 & 3/5 & -4/5\\
1 & 0 & 0 & -3 & 0 & 1 & 1/3 & 0\\
1 & 0 & 0 & -3 & 0 & 1 & -1/3 & 0 
    \end{tabular}
    \caption{Configuration $\mathcal{C}$ of the $24$ lines on the surface $X$ from Example \ref{ex:24_lines}, where in each row we define a line by the intersection of two planes $\sum_{i}h_{i}x_i=\sum_{i}\widetilde{h}_{i}x_i=0$}\label{tab:24_lines}
\end{table}
\label{sec:examples}
\begin{table}[ht]
    \centering
    \tiny
    \begin{tabular}{p{3.5cm}ccccccc}
        $F(x_0,x_1,x_2,x_3)=0$ & source & Reduction? & $p$ & $\mathbb{F}_{p^n}$ & $\#\mathcal{C}$ & $[t_2,t_3,t_4]$ & $E$ \\
        \hline\hline
        $x_0^4-x_0x_1^3-x_2^4+x_2x_3^3$ & \cite{Veniani_lines} & YES & 5 &$\mathbb{F}_{5^2}$ & $64$ & $[336, 64, 8]$ & $29/12\approx 2.42$\\
        \hline
        $x_0^3  x_2 + x_1  x_2^3 + x_1^3  x_3 + x_0  x_3^3$ & \cite{Veniani_lines} & YES & 7 & $\mathbb{F}_{7^4}$ & $52$ & $[ 260, 40, 0]$ & $154/65\approx 2.37$ \\
        \hline
        \shortstack[l]{$3 x_2 x_0^3+3 x_2 x_3 x_0^2-x_2^3 x_0$\\
        $-3 x_1 x_2^2 x_0-3 x_1 x_3^2 x_0$\\
        $-x_1 x_3^3+3 x_1^3x_3+3 x_1^2 x_2 x_3$} & \cite{Veniani_lines} & YES & 5 & $\mathbb{F}_{5^8}$ & 54 & $[ 308, 32, 0]$ & $167/72\approx 2.32$\\
         \hline
        \shortstack[l]{$x_1^3x_2 - x_1 x_2^3 + x_0^3 x_3 - x_0x_3^3 $\\
        $- s x_0^2x_1 x_2$, $s=3$} & \cite{Veniani_lines} & YES & 5 & $\mathbb{F}_{5^8}$ & $26$ & $[51,0,3]$ & $37/16\approx 2.31$ \\
        
        \hline
         \shortstack[l]{$a^4 x_1 x_3^3-a^3 x_1^3 x_3-\left(a^3-2 a\right) x_0 x_1^2 x_2$\\
        $+\left(2 a^3-a\right) x_0^2x_1 x_3-\left(2 a^2-1\right) x_1 x_2^2 x_3 $\\
        $-\left(a^4-2 a^2\right) x_0 x_2 x_3^2-ax_0^3 x_2-x_0 x_2^3$,\\
        $a^4-3a^3+2a^2+3a+1=0$} &  \cite{Veniani_lines} & YES & 5 & $\mathbb{F}_{5^4}$ & $52$ & $[356, 8, 0]$ & $162/73\approx 2.22$\\
        \hline
        \shortstack[l]{
        $\left(-36 p^2+6696 p+2052\right) x_1 x_0^3$\\
        $+\left(11 p^2-5542 p+1121\right) x_2 x_3x_0^2$\\
        $+\left(-540 p^2+6048 p-684\right) x_1^3 x_0$\\
        $+\left(-3919 p^2+2318 p-361\right)x_3^3 x_0$\\
        $+\left(-116 p^2+1612 p-380\right) x_1 x_2 x_3 x_0$\\
        $+3312 p^2 x_1^4$\\
        $+\left(-19p^2-100 p+209\right) x_1 x_2^3$\\
        $+\left(-3331 p^2-100 p+209\right) x_1^2 x_2 x_3$\\
        $+4968 p x_1^2 x_0^2+x_0^4$,\\
        $p^3-201p^2+111p-19=0$} & \cite{Veniani_lines} & YES & 7 &$\mathbb{F}_{7^{30}}$ & $52$ & $[378,0,0]$ & $326/149\approx 2.19$\\
        \hline
        
        \shortstack[l]{$-s x_1 x_0 \left(s x_0 x_2+s x_1 x_3
        +x_1 x_2+x_0 x_3\right)$\\
        $+x_3 x_0^3-x_3^3 x_0-x_1x_2^3+x_1^3 x_2$, $s=3$}& \cite{Veniani_char_3} & YES & 41 & $\mathbb{F}_{41^6}$ & 22 & $[62,0,2]$ & $23/11\approx 2.09$\\
        \hline
        \shortstack[l]{$\left(p^4+1\right) \left(q^4+1\right) \left(x_0^2 x_1^2+x_2^2 x_3^2\right)$\\
        $-2\left(p^4+1\right) q^2 \left(x_0^2 x_2^2+x_1^2 x_3^2\right)$\\
        $-2 p^2 \left(q^4+1\right)
        \left(x_1^2 x_2^2+x_0^2 x_3^2\right)$\\
        $+2 p^2 q^2\left(x_0^4+x_1^4+x_2^4+x_3^4\right)$, $p=2,q=3$} & \cite{Veniani_lines} & NO & $0$ & $\mathbb{Q}$ & $24$ & $[96,0,0]$ & 2\\
        
    \end{tabular}
    \caption{For each smooth quartic we find a prime of good reduction $p$ and compute the maximal configuration of lines and the count of intersection points of given multiplicity $r$. In the last column we compute the slope of the configuration.}
    \label{tab:quartic_examples}
\end{table}

\end{document}